\documentclass[11pt,reqno,a4paper]{amsart}

\usepackage{mathrsfs, graphics}
\usepackage{amssymb}
\usepackage{amsfonts}
\usepackage{enumitem} % customize enumerates
\usepackage{hyperref}
\usepackage{graphicx, epstopdf}\graphicspath{{./immaginiprova/}}
\usepackage{color}
\usepackage[mathscr]{euscript}
\numberwithin{equation}{section}

\theoremstyle{plain} %to change the headings as suggested by Kleiner
\newtheorem{lem}{Lemma}[section]
\newtheorem{prop}[lem]{Proposition}
\newtheorem{thm}[lem]{Theorem}
\newtheorem{cor}[lem]{Corollary}
\newtheorem{question}[lem]{Question}
\newtheorem{remark}[lem]{Remark}
\theoremstyle{definition}
\newtheorem{defn}[lem]{Definition}

\newtheoremstyle{mytheorem}
{}% measure of space to leave above the theorem. E.g.: 3pt
{}% measure of space to leave below the theorem. E.g.: 3pt
{\it}% name of font to use in the body of the theorem
{\parindent}% measure of space to indent
{\bf}% name of head font
{.}% punctuation between head and body
{ }% space after theorem head; " " = normal interword space
{\thmnumber{#2.~}\thmname{#1}\thmnote{~\rm#3}}% Manually specify head

\newtheoremstyle{myremark}
{}% measure of space to leave above the theorem. E.g.: 3pt
{}% measure of space to leave below the theorem. E.g.: 3pt
{\rm}% name of font to use in the body of the theorem
{\parindent}% measure of space to indent
{\bf}% name of head font
{.}% punctuation between head and body
{ }% space after theorem head; " " = normal interword space
{\thmnumber{#2.~}\thmname{#1}\thmnote{~\rm#3}}% Manually specify head

\newtheoremstyle{myparagraph}
{}% measure of space to leave above the theorem. E.g.: 3pt
{}% measure of space to leave below the theorem. E.g.: 3pt
{\rm}% name of font to use in the body of the theorem
{\parindent}% measure of space to indent
{\bf}% name of head font
{.}% punctuation between head and body
{ }% space after theorem head; " " = normal interword space
{\thmnumber{#2.~}\thmname{#1}\thmnote{#3}}% Manually specify head

\makeatletter
\def\@secnumfont{\sc}
\def\section{\@startsection{section}{1}%
\z@{1.5\linespacing\@plus .2\linespacing}{.7\linespacing}%
{\normalfont\sc\centering}}
\makeatother
\makeatletter
\def\ps@headings{\ps@empty
 \def\@evenhead{%
  \setTrue{runhead}%
  \normalfont\footnotesize
  \rlap{\thepage}\hfil
  \def\thanks{\protect\thanks@warning}%
  \leftmark{}{}\hfil}%
 \def\@oddhead{%
  \setTrue{runhead}%
  \normalfont\footnotesize\hfil
  \def\thanks{\protect\thanks@warning}%
  \rightmark{}{}\hfil \llap{\thepage}}%
\let\@mkboth\markboth}
\makeatother
\makeatletter
\renewenvironment{proof}[1][\proofname]{\par
  \pushQED{\qed}%
  \normalfont \topsep6\p@\@plus6\p@\relax
  \trivlist
  \itemindent\normalparindent
  \item[\hskip\labelsep
    \bfseries
    #1\@addpunct{.}]\ignorespaces
}{%
  \popQED\endtrivlist\@endpefalse
}
\providecommand{\proofname}{Proof}
\makeatother
\newcommand{\R}{\mathbb{R}}
\newcommand{\N}{\mathbb{N}}

\newcommand{\Leb}{\mathscr{L}}

\newcommand{\Lip}{\mathrm{Lip}}

\newcommand{\Tan}{\mathrm{Tan}}
\newcommand{\Int}{\mathrm{Int}}
\newcommand{\dist}{\mathrm{dist}}

\newcommand{\dV}{d_V\kern-1pt}
\newcommand{\trait}[3]{\vrule width #1ex height #2ex depth #3ex}
\newcommand{\trace}{\mathchoice%
  {\mathbin{\trait{.12}{1.2}{.03}\trait{.8}{0.09}{0.03}}}
  {\mathbin{\trait{.12}{1.2}{.03}\trait{.8}{0.09}{0.03}}}
  {\mathbin{\hskip.15ex\trait{.09}{.84}{0.02}\trait{.56}{.07}{.02}}\hskip.15ex}
  {\mathbin{\trait{.07}{.6}{.01}\trait{.4}{.06}{.01}}}}
\makeatletter
\@addtoreset{footnote}{section}
\makeatother
\def\glip#1.{{\bf L}(#1)} %global Lipschitz constant
\def\glipdec#1,#2.{{\bf L}_{#2}(#1)}
\def\tang#1.{\setbox1=\hbox{$#1$\unskip}
{\normalfont\text{T}}_{\ifdim\wd1>0pt #1\else x,r_a\fi}}
\def\pcreatedst#1#2{\expandafter\def\csname #1dst\endcsname##1,##2.{#2(##1,##2)}
\expandafter\def\csname #1dstp\endcsname##1.{{#2}_{##1}}
\expandafter\def\csname #1dname\endcsname{#2}} % creates a
\def\pcreatenrm#1#2#3{\expandafter\def\csname
  #1nrm\endcsname##1.{\left #3 ##1\right #3_{#2}} \expandafter\def\csname
  #1nrmname\endcsname{\left #3\,\cdot\,\right #3_{#2}}} % creates a norm
\def\pcreateasymnrm#1#2#3#4{\expandafter\def\csname
  #1nrm\endcsname##1.{\left #3 ##1\right #4_{#2}} \expandafter\def\csname
  #1nrmname\endcsname{\left #3\,\cdot\,\right #4_{#2}}} % creates a norm
\def\wder#1.{{\mathscr{X}}({#1})}%L^\infty module of derivations
\def\wform#1.{{\mathscr{E}}({#1})}%L^\infty module of forms
\def\zahlen{{\mathbb Z}}
\def\real{{\mathbb{R}}}
\def\ball#1,#2.{B(#1,#2)} % ball: #1 is the centre, #2 is the radius
\def\clball#1,#2.{\bar B(#1,#2)} % closed ball: #1 is the centre, #2
\def\sball#1,#2,#3.{B_{#1}(#2,#3)} % subscripted ball: #1 is
\def\scball#1,#2,#3.{\bar B_{#1}(#2,#3)} % closed subscripted ball: #1 is
\def\glip#1.{{\bf L}(#1)} %global Lipschitz constant
\def\glipdec#1,#2.{{\bf L}_{#2}(#1)}
\def\pcreatedst#1#2{\expandafter\def\csname #1dst\endcsname##1,##2.{#2(##1,##2)}
\expandafter\def\csname #1dstp\endcsname##1.{{#2}_{##1}}
\expandafter\def\csname #1dname\endcsname{#2}} % creates a
\def\pcreatenrm#1#2#3{\expandafter\def\csname
  #1nrm\endcsname##1.{\left #3 ##1\right #3_{#2}} \expandafter\def\csname
  #1nrmname\endcsname{\left #3\,\cdot\,\right #3_{#2}}} % creates a norm
\def\pcreateasymnrm#1#2#3#4{\expandafter\def\csname
  #1nrm\endcsname##1.{\left #3 ##1\right #4_{#2}} \expandafter\def\csname
  #1nrmname\endcsname{\left #3\,\cdot\,\right #4_{#2}}} % creates a norm
\def\wder#1.{{\mathscr{X}}({#1})}%L^\infty module of derivations
\def\wform#1.{{\mathscr{E}}({#1})}%L^\infty module of forms

\def\mpush#1.{{#1}_{\sharp}} %pusch forward of a measure
\def\lipalg#1.{{\rm Lip}_{\text{\normalfont b}}(#1)} %Lipschitz
\def\albrep#1.{{\mathcal A}_{#1}} %symbol for an Alberti
\def\ccompact#1.{{C_c(#1)}}

\newcommand{\on}{\:\mbox{\rule{0.1ex}{1.2ex}\rule{1.1ex}{0.1ex}}\:}
\DeclareMathOperator\frags{Frag} %fragments
\DeclareMathOperator\cl{cl} %fragments
\DeclareMathOperator\dom{dom} % domain
\def\lebmeas#1.{\setbox1=\hbox{$#1$\unskip}{\mathcal L}^{\ifdim\wd1>0pt
    #1 \else 1 \fi}} %lebesgue measure
\DeclareMathOperator\spt{spt} %support of a function, measure, etc...
\def\locnorm#1,#2.{\left|{#1}\right|_{{#2},\text{\normalfont loc}}}%local norm of
\def\lipfun#1.{{\rm Lip}(#1)} %Lipschitz functions
\def\makestep#1#2{\par\noindent\texttt{Step#1: #2.}}
\begin{document}

	% The following instructions defines the headings
	% for all pages except the first one
	% They can be safely removed
	%
\pagestyle{empty}
\pagestyle{myheadings}
\markboth%
{\underline{\centerline{\hfill\footnotesize%
\textsc{Andrea Marchese, Andrea Schioppa}%
\vphantom{,}\hfill}}}%
{\underline{\centerline{\hfill\footnotesize%
\textsc{Prescribed blowups}%
\vphantom{,}\hfill}}}

	% We do not use the maketitle command, and
	% in the next lines we define the headline for the
	% first page, then title, authors' names, and so on...
	% acknowledgements are at the end of the introduction
	% affiliations are at the end of the paper
	%
\thispagestyle{empty}

~\vskip -1.1 cm

	% heading of first page
	%
%{\footnotesize\noindent
%[version:~\dataversione]%
%\hfill
%%to apper on\dots
%%\hfill DOI~\href{http://dx.doi.org/***}{***} \par
%}

\vspace{1.7 cm}

	% title
	%
{\Large\sl\centering
Lipschitz functions with prescribed blowups at many points
\\
}
\vspace{.6 cm}

	% authors' names
	%
\centerline{\sc Andrea Marchese and Andrea Schioppa}

\vspace{.8 cm}

{\rightskip 1 cm
\leftskip 1 cm
\parindent 0 pt
\footnotesize

	% abstract, keywords and MSC numbers
	%
{\sc Abstract.}
In this paper we prove generalizations of Lusin-type theorems for
gradients due to Giovanni Alberti, where we replace the Lebesgue
measure with any Radon measure $\mu$. We apply this to go beyond the
known result on the existence of Lipschitz functions which are
non-differentiable at $\mu$-almost every point $x$ in any direction
which is not contained in the decomposability bundle $V(\mu,x)$,
recently introduced by Alberti and the first author. More
precisely, we prove that it is possible to construct a Lipschitz
function which attains any prescribed \emph{admissible blowup} at
every point except for a closed set of points of arbitrarily small
measure. Here a function is an admissible blowup at a point $x$ if it
is null at the origin and it is the sum of a linear function on
$V(\mu,x)$ and a Lipschitz function on $V(\mu,x)^{\perp}$.

\par
}

%%%%%%%%%%%%%%%%%%%%%%%%%%%%%%%%%%%%%%%%%%%%%%%%%%%%%%%%%%%%%%%%%
%
%	INTRODUCTION
%
%%%%%%%%%%%%%%%%%%%%%%%%%%%%%%%%%%%%%%%%%%%%%%%%%%%%%%%%%%%%%%%%%

\section{Introduction}
% \begin{itemize}
% \item[-]Teorema di Alberti.
% \item[-]Osservazione che il teorema di Rademacher non permette di prescrivere altri tipi di blowup oltre a quelli lineari, per funzioni Lipschitziane.
% \item[-]Definizione di $Tan(g,x)$ insieme dei blowup di una funzione Lipschitzina $g$ in un punto. Forse per comodit\`a definirei le funzioni riscalate e i blowup solo sulla palla unitaria.
% \item[-] Definizione provvisoria del ``fibrato dei blowup'' $C(\mu,x)$. Se non riusciamo a migliorare il risultato, per ora definirei $C(\mu,x)$ come una mappa boreliana da $\R^N$ a $Lip(B,\R)$ tale che 
% $$C(\mu,x):=\{f:B\to\R: f\;{\rm{Lipschitz}}, f(y)=L(y_V)+H(y_{V^\perp})\},$$
% dove $V$ \`e il fibrato di decomponibilit\`a e $y_V,y_{V^\perp}$ sono le proiezioni di $y$ su $V$ e sull'ortogonale rispettivamente, $L$ \`e lineare, $H$ \`e Lipschitz.
% \end{itemize}  

In \cite{alberti-lusin}, Alberti proved a ``Lusin type theorem for gradients'':
roughly speaking, given any Borel vectorfield $f$ on the Euclidean
space $\R^N$ one can find a $C^1$ function $g$ whose gradient
coincides with $f$ up to an exceptional open set of arbitrarily
small Lebesgue measure. In other words one can prescribe at many
points the (unique) blowup of a $C^1$ function in an arbitrary
(measurable) way. Rademacher Theorem, which states that Lipschitz
functions are differentiable almost everywhere, implies that, even if
one weakens the assumptions on $g$, requiring it only to be locally Lipschitz, Alberti's result is still the best possible: no other blowups than the linear ones can be prescribed on a set of points of positive measure; moreover, one cannot get rid of the small exceptional set without any further assumption on $f$, i.e.~in general one cannot find a Lipschitz function $g$ such that the measure of the set $\{Dg\neq f\}$ is zero. However a continuous function $g$ with such property can be found (see \cite{moonens-lusin}).

In the present paper, we prove a generalization of Alberti's result,
where the Lebesgue measure is replaced by any Radon measure. Since
Rademacher Theorem does not hold in general with respect to a Radon
measure and in particular it fails with respect to any singular
measure (see Theorem 1.14 of \cite{rindler-afree}), then the following
vague question is very natural in our setting. Given a measure $\mu$
on $\R^N$, which blowups is it possible to prescribe for a Lipschitz
function, at many points with respect to $\mu$, besides the linear ones?\\

Let us introduce some basic notations to make the question more precise. We denote by $B(x,r)$ the ball with center $x\in\R^N$ and radius $r>0$. We simply write $B$ for the unit ball centred at the origin.

\begin{defn}[Blowups of a Lipschitz function]\label{def_bu}
Given a Lipschitz function $g$ defined on an open subset $\Omega\subset\R^N$, and a point $x\in \Omega$, we denote by $\Tan(g,x)$ the set of all the possible limits, with respect to the uniform convergence,
$$\lim_{j\to\infty}\tang {x,r_j}.f,$$
where $r_j\searrow 0$ and for every $r\leq \dist(x,\Omega^c), \tang {x,r}.f:B\to\R$ is defined by
$$\tang {x,r}.f(y):=r^{-1}(f(x+ry)-f(x)), \quad\text{for every }y\in B.$$

\end{defn}

\begin{defn}[Prescribing blowups]\label{def_presc}
Let $\mu$ be a positive Radon measure on an open set $\Omega\subset \R^N$. Denote by ${\rm{Lip}}(B,0)$ the space of Lipschitz functions on $B$ which vanish at the origin, endowed with the supremum distance, and let $f:\Omega\subset\R^N\to {\rm{Lip}}(B,0)$ be a Borel function. We say that $f$ {\it{prescribes the blowups}} of a Lipschitz function with respect to $\mu$:
\begin{itemize}
\item[(i)] {\it{Weakly}}, if there exists a Lipschitz function $g:\Omega\to\R$ such that $f(x)\in \Tan(g,x)$ for $\mu$-a.e. $x\in\Omega$;
\item[(ii)] {\it{Weakly in the Lusin sense}}, if for every $\varepsilon>0$ there exists a Lipschitz function $g:\Omega\to\R$ such that $$\mu(\{x\in\Omega:f(x)\not\in \Tan(g,x)\})<\varepsilon;$$
\item[(iii)] {\it{Strongly in the Lusin sense}}, if for every $\varepsilon>0$ there exists a Lipschitz function $g:\Omega\to\R$ such that $$\mu(\{x\in\Omega:\{f(x)\}\neq \Tan(g,x)\})<\varepsilon.$$
\end{itemize}  
\end{defn}

In this paper we mainly address the following question.
\begin{question}\label{q:1}
Given a positive Radon measure $\mu$ on an open set $\Omega\subset\R^N$ with $\mu(\Omega)<\infty$, for which choice of $f$ is it possible to say that $f$ prescribes the blowups of a Lipschitz function wrt $\mu$ weakly/strongly/in the Lusin sense?
\end{question}

As we already observed, when $\mu$ is the Lebesgue measure, Rademacher Theorem is a constraint on the possible choices of a function $f$ for which Question \ref{q:1} may have a positive answer. Namely, in this case, for a.e. point $x$, the corresponding function $f(x)$ must be linear (more precisely, the restriction to $B$ of a linear function). Using the notation that we have introduced above, the content of Alberti's result, or at least part of it, can be rephrased as follows: if $\Omega\subset \R^N$ is an open set with finite Lebesgue measure, then every Borel function $f:\Omega\to{\rm{Lip}}(B,0)$ whose values are (restrictions to $B$ of) linear functions almost everywhere, prescribes the blowups of a Lipschitz function wrt the Lebesgue measure strongly in the Lusin sense.

For a general measure $\mu$ Rademacher Theorem does not hold. In
particular De Philippis and Rindler in \cite{rindler-afree} completed the proof, also based on other works, that there are Lipschitz functions which are non-differentiable at $\mu_{\text{sing}}$-a.e. point, where $\mu_{\text{sing}}$ is the singular part of $\mu$ wrt Lebesgue. Nevertheless a suitable weaker version of Rademacher Theorem holds. Indeed, letting ${\rm{Gr}}(\R^N)$ denote the union of the Grasmannians of all vector subspaces of $\R^N$, in \cite{alberti-marchese} it is proved that to every Radon measure $\mu$ on $\R^N$ it is possible to associate a Borel function $V(\mu,\cdot):\R^N\to{\rm{Gr}}(\R^N)$ called the \emph{decomposability bundle} of $\mu$, with the property that for every Lipschitz function $g$, the restriction of $g$ to the affine subspace $x+V(\mu,x)$ is differentiable at $\mu$-a.e. point $x$ and moreover the bundle is maximal with respect to this property, meaning that there exists a Lipschitz function which is non-differentiable at $\mu$-a.e. point $x$ along any direction which is not in $V(\mu,x)$. In virtue of \cite[Theorem 1.1 (ii)]{alberti-marchese}, the proof that Rademacher Theorem does not hold for singular measures reduces to proving that every singular measure on $\R^N$ has decomposability bundle of dimension at most $N-1$. This is achieved in \cite{rindler-afree} using a characterization of the decomposability bundle given in \cite[Theorem 6.4]{alberti-marchese}.

Clearly the main result of \cite{alberti-marchese} is also a constraint on the possible choices of a function $f$ for which one can expect a positive answer to Question \ref{q:1}, indeed one should at least require that $f(x)$ is linear on $V(\mu,x)$ for $\mu$-a.e. point $x$. This observation partially motivates the introduction of the following subset of ${\rm{Lip}}(B,0)$. Given a vector subspace $V$ of $\R^N$ and a point $y\in\R^N$ we denote respectively $y_V$ and $y_{V^\perp}$ the projections on $V$ and on its orthogonal complement $V^\perp$. Finally we denote the class of \emph{admissible blowups} by
\begin{equation}\label{e:admissible}
C(\mu,x):=\{h\in{\rm{Lip}}(B,0):h(y)=L(y_{V(\mu,x)})+m(y_{V(\mu,x)^\perp})\},
\end{equation}
where $L$ is a linear function on $V(\mu,x)$ and $m$ is a Lipschitz
function on $V(\mu,x)^\perp$. 

Now we are ready to state the main results of the paper.

\begin{thm}\label{main}
Let $\mu$ be a Radon measure on $\R^N$, let $\Omega\subset\R^N$ be an open set with $\mu(\Omega)<\infty$, and let $f$ be as in Definition \ref{def_presc}. Then the following statements hold:
\begin{itemize}
\item[(I)] if $f(x)=L(x)$ at $\mu$-a.e $x$, where $L(x)$ is the restriction to $B$ of a linear function, then $f$ prescribes the blowups of a Lipschitz function with respect to $\mu$ strongly in the Lusin sense;
\item[(II)] if $f(x)\in C(\mu,x)$ for $\mu$-a.e. $x$, then $f$ prescribes the blowups of a Lipschitz function with respect to $\mu$ weakly in the Lusin sense.
\end{itemize}
\end{thm}
In Section \ref{s6} we exhibit a measure $\mu$ for which one cannot prescribe more blowups that those contained in the class $
C(\mu,\cdot)$, proving the sharpness of (II). For $N=1$ we can prove a stronger statement. In particular we don't need the restriction that $\Omega$ has finite measure. Firstly we show that the only measures for which one can prescribe strongly some non-linear blowups are the atomic ones. Secondly we prove that any blowup can be prescribed weakly wrt a singular measure $\mu$. More precisely, for the typical $1$-Lipschitz function $g$ (in the sense of Baire categories), $\Tan(g,x)$ coincides with the set of all $1$-Lipschitz functions in ${\rm{Lip}}(B,0)$, at $\mu$-a.e. point $x$.\\ 

Given a Borel set $E$, we denote by $\mu\trace E$ the measure defined by
$$\mu\trace E(A):=\mu(A\cap E),$$
for every Borel set $A$.

\begin{thm}\label{main1}
Let $\Omega\subset\R$ be an open set, and let $\mu$ and $f$ be as in Definition \ref{def_presc}. Then the following statements hold:
\begin{itemize}
\item[(I)] $f$ prescribes the blowups of a Lipschitz function with respect to $\mu$ strongly in the Lusin sense, if and only if $f(x)$ is the restriction to $B$ of a positively homogeneous function, for $\mu$-a.e. $x$ and $\mu\trace NL$ is atomic, where $$NL:=\{x\in\Omega: f(x)\; \mbox{is not the restriction to}\; B\; \mbox{of a linear function}\};$$
\item[(II)] if $\mu$ is singular, then $f$ prescribes the blowups of a Lipschitz function with respect to $\mu$ weakly;
\end{itemize}
\end{thm}

\begin{remark}\label{rmk1}\rm{
\begin{itemize}
\item[(i)] Statement (I) in Theorem \ref{main} is a generalization of
  Theorem 1 of \cite{alberti-lusin}. In Section \ref{s3} we prove a
  more precise version of this statement, including the possibility to
  choose the Lipschitz function $g$ in point (iii) of Definition
  \ref{def_presc} of class $C^1$, with arbitrarily small $L^{\infty}$
  norm and with $L^p$ estimates on its gradient for every
  $p\in[1,\infty]$. A similar result was recently proved by David in
  \cite{david-lusin} in the setting of PI spaces: a class of metric
  measure spaces which admit a differentiable structure. We point out
  that statement (I) can also be extended to doubling metric measure
  spaces, where the differentiable structure is defined using
  operators called derivations: we will not pursue this issue in the present paper.
\item[(ii)] The difference between statement (II) of Theorem \ref{main1} and statement (II) of Theorem \ref{main} is twofold. Firstly, in the 1-dimensional case there is no restriction on the function $f$, due to the fact that the decomposability bundle of a singular measure in $\R$ is always trivial. Secondly, we remark that the blowups in the 1-dimensional case are prescribed weakly, while in the general case are prescribed only weakly in the Lusin sense. More precisely, in dimension $N=1$ we are able to prove that residually many 1-Lipschitz functions attain, in a set of full measure, every 1-Lipschitz function in ${\rm{Lip}}(B,0)$ as a blowup. 
%In higher dimension, for fixed $\varepsilon$ it would be possible to prove the residuality of the set of Lipschitz functions attaining every ``admissible'' blowup outside a set of measure $\varepsilon$, with a technique similar to that used in the proof of Theorem 1.1 (ii) of \cite{alberti-marchese}. The issue is that with such technique one could only prove that the above set is residual in a complete metric space of Lipschitz functions, which is suitably defined depending on $\varepsilon$ itself. Hence, the intersection of a countable family of such sets, might in principle be empty.
\item[(iii)] Statement (I) of Theorem \ref{main1} is a simple observation, which is already contained in Proposition 4.2 of \cite{Ma}. The restriction to the family of positively homogeneous functions is clearly necessary in order to prescribe blowups strongly, because if $f\in \Tan(g,x)$ and $h\in \Tan(f,0)$, then it also holds $h\in \Tan(g,x)$. Presumably, also in dimension larger than 1 the possibility to prescribe strongly some non-linear blowups in the Lusin sense should depend on some property of the measure and intuitively it should fail when the measure is ``very diffused''. On the other side, it sounds reasonable that if the measure $\mu$ is supported on a $k$-rectifiable set $E$ in $\R^N$ ($k<N$) then any Borel function $f$ which is $\mu$-a.e. positively homogeneous and linear along the tangent bundle to $E$ prescribes the blowups of a Lipschitz function strongly in the Lusin sense. However we do not pursue this issue in this paper.
\end{itemize}}
\end{remark}

\subsection*{On the structure of the paper} 
The proof of Theorem \ref{main} is split in Section \ref{s3} for statement (I) and Section \ref{s5}, for statement (II). In Section \ref{s6} we provide an example of a measure $\mu$ for which every blowup of a Lipschitz function is the sum of a linear function on $V(\mu,x)$ and a Lipschitz function on its orthogonal, at $\mu$-a.e. point $x$, in order to justify the choice of the class $C(\mu,\cdot)$ appearing in statement (II) of Theorem \ref{main}. Finally, in Section \ref{s4}, we prove Theorem \ref{main1}. 

\subsection*{Acknowledgements} 
The authors would like to thank Giovanni Alberti for several discussions. A. M. was supported by the ERC-grant ``Regularity of
area-minimizing currents'' (306247). A.S. was supported
by the ``ETH Zurich Postdoctoral Fellowship Program and the Marie
Curie Actions for People COFUND Program''.

\section{Proof of Theorem \ref{main}(I)}\label{s3}
In this section we prove statement (I) of Theorem \ref{main}. As anticipated in Remark \ref{rmk1} (i) we will actually prove a stronger statement including some gradient estimates. In particular $L^{\infty}$ estimates on the function $g$ of Definition \ref{def_presc} and its gradient are necessary to prove part (II) of Theorem \ref{main}. The proof of statement (I) is very similar to the one presented in \cite{alberti-lusin}. The new main technical ingredient is Corollary \ref{cor:neg_frames}.
\begin{thm}\label{mainI}
Let $\mu$ be a Radon measure on $\R^N$. Let $\Omega\subset\R^N$ open with $\mu(\Omega)<\infty$. Then for every Borel map $f:\Omega\to \R^N$ and for every
  $\varepsilon, \zeta>0$ there exist a compact set $K$ and a function $g\in C^1_c(\Omega)$ with $\|g\|_{\infty}\leq\zeta$ such
  that
\begin{equation}\label{e_tmain1}
\mu(\Omega\setminus K)<\varepsilon\mu(\Omega).
\end{equation}
\begin{equation}\label{e_tmain2}
Dg(x)=f(x), \quad\text{for every } x\in K.
\end{equation}
Moreover, there exists $C=C(N)$ such that
\begin{equation}\label{e_tmain3}
\|Dg\|_p\leq C\varepsilon^{1/p-1}\|f\|_p,\quad\text{for every } p\in[1,\infty],
\end{equation}
where $\|\cdot\|_p$ denotes the usual norm in $L^p(\Omega,\mu)$.
\end{thm}
%%%%%%%%%%%%%%%%%%%%%%%%%%%%%%%%%%%%%%%%%%%%%%%%%%%%%%%%%%%%%%%%%%%%%%%%%%
%% @def
\def\cubes#1.{{\text{\normalfont
      B}}_{\zahlen}\setbox1=\hbox{$#1$\unskip}
\ifdim\wd1>0pt (#1)\fi}
\def\bxset#1,#2.{\text{\normalfont Bx}\left(
    \setbox1=\hbox{$#1$\unskip}\setbox2=\hbox{$#2$\unskip}
    \ifdim\wd1>0pt #1 \else x\fi,
\ifdim\wd2>0pt#2 \else r\fi\right)
} % box set
\def\fxset#1,#2,#3.{\text{\normalfont Fr}\left(
    \setbox1=\hbox{$#1$\unskip}\setbox2=\hbox{$#2$\unskip}
    \setbox3=\hbox{$#3$\unskip}
    \ifdim\wd1>0pt #1 \else x\fi,
\ifdim\wd2>0pt#2 \else r\fi,\ifdim\wd3>0pt#3 \else \varepsilon\fi\right)
} % frame set
%%%%%%%%%%%%%%%%%%%%%%%%%%%%%%%%%%%%%%%%%%%%%%%%%%%%%%%%%%%%%%%%%%%%%%%%%%
Let $\cubes.$ denote the collection of boxes in $\real^N$ of the form
$\prod_{i=1}^N[2n_i-1, 2n_i+1]$ for $(n_i)_{i=1}^N\in\zahlen^N$. For
$r>0$ let $\cubes r.$ denote the transform of $\cubes.$ when the
dilation $x\mapsto rx$ is applied to $\real^N$.
For $x\in\real^N$ and $r>0$ we denote the box with center $x$ and edge length $2r$ by
\begin{equation}
  \label{eq:bxdef}
  \bxset ,.:=\left\{y\in\real^N:\max_{1\le i\le N}|x_i-y_i|\le r\right\},
\end{equation}
and for $0<\varepsilon<1$ we define the ``frame'':
\begin{equation}
  \label{eq:fxdef}
  \fxset ,,.:=\left\{ y\in\bxset,.:\text{
      $|x_i-y_i|\ge(1-\varepsilon)r$ for some $i\in\{1,\cdots,N\}$}
\right\}.
\end{equation}
Finally we consider the probability measure $P:=\frac{1}{(2r)^N}\lebmeas N.\on\bxset 0,.$ and for $\omega\in\bxset 0,r.$ we let $\cubes r,\omega.:=\cubes
r.+\omega$.
%%%%%%%%%%%%%%%%%%%%%%%%%%%%%%%%%%%%%%%%%%%%%%%%%%%%%%%%%%%%%%%%%%%%%%%%%%
% @def
\def\gdec{_{\varepsilon,\omega}^{\normalfont\text{good}}}
\def\bdec{_{\varepsilon,\omega}^{\normalfont\text{bad}}}
\def\bodec{_{\normalfont\text{bad}}}
%%%%%%%%%%%%%%%%%%%%%%%%%%%%%%%%%%%%%%%%%%%%%%%%%%%%%%%%%%%%%%%%%%%%%%%%%%
\begin{lem}[Existence of boxes with negligible frames]
  \label{lem:neg_frames}
  Let $K$ be compact with $\mu(K)>0$ and $U\supset K$ open with
  $\mu(U)\le\frac{3}{2}\mu(K)$. Assume that $r>0$ is such that for
  each $\bxset ,.$ which intersects $K$ one has $\bxset ,.\subset
  U$. For $\varepsilon>0$ and $\omega\in\bxset 0,r.$ define
  \begin{equation}
    \label{eq:neg_frames_s1}
\begin{split}
    B\gdec := \biggl\{
        \bxset ,.\in&\cubes r,\omega.: \bxset ,.\cap K\ne\emptyset\\
&\text{and $\mu(\fxset ,,.) \le 16N2^N\varepsilon\mu(\bxset ,.)$}
\biggr\}.
\end{split}
\end{equation}
Then for some $\omega\in\bxset 0,r.$ one has:
\begin{equation}
  \label{eq:neg_frames_s2}
  \mu(\bigcup B\gdec\cap K)\ge\frac{3}{4}\mu(K).
\end{equation}
\end{lem}
%%%%%%%%%%%%%%%%%%%%%%%%%%%%%%%%%%%%%%%%%%%%%%%%%%%%%%%%%%%%%%%%%%%%%%%%%%
\begin{proof}
  We define the $\varepsilon$-boundaries of a family of boxes $G$
  \begin{equation}
    \label{eq:neg_frames_p1}
    \partial_\varepsilon G := \left\{
\fxset,,.:\bxset,.\in G
\right\},
\end{equation}
the set of bad boxes
\begin{equation}    \label{eq:neg_frames_p2}
\begin{split}
    B\bdec := \biggl\{
        \bxset ,.\in&\cubes r,\omega.: \bxset ,.\cap K\ne\emptyset\\
&\text{and $\mu(\fxset ,,.) > 16N2^N\varepsilon\mu(\bxset ,.)$}
\biggr\},
\end{split}
\end{equation}
and the set of bad $\omega$'s:
\begin{equation}
    \label{eq:neg_frames_p3}
A\bodec := \left\{\omega\in\bxset 0,r.: \mu(\bigcup B\bdec)\ge\frac{1}{6}\mu(U)\right\}.
\end{equation}  
The Lemma is proven by showing that $P(A\bodec)<1$. Define:
\begin{equation}
  \label{eq:neg_frames_p4}
  I := \int \chi_{\bigcup\partial_\varepsilon B\bdec}(x)\chi_{A\bodec}(\omega)\,d\mu(x)\,dP(\omega),
\end{equation}
where $\chi_E$ denotes the characteristic function of the set $E$, with values $0$ and $1$. We estimate $I$ from below integrating first in $d\mu(x)$:
\begin{equation}
\begin{split}
  \label{eq:neg_frames_p5}
  I &= \int\mu(\bigcup\partial_\varepsilon
  B\bdec)\chi_{A\bodec}(\omega)dP(\omega)\\
  &\ge 16N2^N\varepsilon\int\mu(\bigcup
  B\bdec)\chi_{A\bodec}(\omega)dP(\omega)\\
  &\ge\frac{16N2^N\varepsilon}{6}\mu(U)P(A\bodec).
\end{split}
\end{equation}
We estimate $I$ from above integrating first in $dP(\omega)$:
\begin{equation}
  \label{eq:neg_frames_6}
  I = \int P(\omega\in A\bodec: x\in\bigcup\partial_\varepsilon
  B\bdec)\,d\mu(x).
\end{equation}
For fixed $x$ the set $\{\omega:x\in\bigcup\partial_\varepsilon
  B\bdec\}$ has positive $P$-measure only if $x\in U$ and
  the one must also have $\omega\in\bigcup_{j=1}^{2^N}\fxset \tilde x_j,,.$ where the $\{\tilde x_j\}$ depend
  only on $x$. As the Lebesgue measure on ${\fxset \tilde
    x_j,r,\varepsilon.}$ is at most $2N(2r)^N\varepsilon$
  we get
  \begin{equation}
    \label{eq:neg_frames_p7}
    I \le \varepsilon(2N)2^N\mu(U)
  \end{equation}
and so $P(A\bodec)\le\frac{12}{16}<1$. Finally for $\omega\in A\bodec^c$ we
observe:
\begin{equation}
  \label{eq:neg_frames_p8}
  \mu(K\cap\bigcup B\gdec)\ge \mu(K)-\mu(\bigcup B\bdec)\ge\frac{3}{4}\mu(K).
\end{equation}
\end{proof}
%%%%%%%%%%%%%%%%%%%%%%%%%%%%%%%%%%%%%%%%%%%%%%%%%%%%%%%%%%%%%%%%%%%%%%%%%%

By a standard covering argument, we deduce the following

\begin{cor}[Covering by good boxes]
\label{cor:neg_frames}
Let $\varepsilon>0$ and $r_0>0$; then for every open set $U\subset\R^N$ there is a sequence of
disjoint boxes $\{\bxset z_\lambda,r_\lambda.\}_\lambda$ contained in $U$ such that:
\begin{align}
  \label{eq:neg_frames_cs1}
  r_\lambda&\le r_0\\  \label{eq:neg_frames_cs2}
 \mu(\fxset z_\lambda,r_\lambda,\varepsilon.)&\le 16N2^N\varepsilon\mu(\bxset
  z_\lambda,r_\lambda.)\\  \label{eq:neg_frames_cs3}
 \mu(U\setminus\bigcup_\lambda\bxset z_\lambda,r_\lambda.)&=0.
\end{align}
\end{cor}

To prove Theorem \ref{mainI} it is sufficient to perform a straightforward iteration of Lemma \ref{lemma_base} below. For the proof of the lemma, after we have established Corollary \ref{cor:neg_frames}, we can easily adapt the proof given in \cite{alberti-lusin}.

\begin{lem}\label{lemma_base}
Let $\Omega$ be an open subset of $\R^N$ with finite measure $\mu$. Let $f:\Omega\to\R^N$ be a bounded and continuous function. Then for every $\xi,\eta, \zeta>0$ there exists a compact set $K\subset\Omega$ and a function $g\in C^1_c(\Omega)$ such that
\begin{equation}\label{e_lem1}
\mu(\Omega\setminus \Int(K))<\xi\mu(\Omega),
\end{equation}
\begin{equation}\label{e_lem1bis}
\|g\|_\infty\leq \zeta,
\end{equation}
\begin{equation}\label{e_lem2}
|f(x)-Dg(x)|\leq\eta, {\rm{for\;every\;}} x\in K.
\end{equation}
Moreover there exists $C=C(N)$ such that
\begin{equation}\label{e_lem3}
\|Dg\|_p\leq C\xi^{1/p-1}\|f\chi_{\spt(g)}\|_p,\;{\rm{for\;every\;}}p\in[1,\infty],
\end{equation}
where $\spt(g)$ is the support of the function $g$.
\end{lem}
\begin{proof}
Suppose $\xi<1$. Let $K'$ be a compact subset of $\Omega$ such that 

\begin{equation}\label{e_k_1}
\mu(\Omega\setminus K')<\mu(\Omega)\xi/3.
\end{equation}
Let $d':=\dist(K',\R^N\setminus\Omega)$ and $d:=\min\{1,d'/2\}$. Denote by $K''$ the compact set
$$K'':=\{x\in\Omega:\dist(x,K')\}\leq d.$$
Since $f$ is uniformly continuous in $K''$, there exists $0<\delta<d$ such that for all $x\in K''$, $y\in\Omega$ it holds
\begin{equation}\label{e_unif_cont}
|x-y|<\delta\implies|f(x)-f(y)|<\eta.
\end{equation}

Consider the family of boxes $\{\bxset x_i,r_i.\}_i$ obtained applying
Corollary \ref{cor:neg_frames}, with $U=\Omega$ and the choice of parameters
\begin{equation}\label{e:choice_r_0}
r_0=\min\left\{\frac{\delta}{2N};\frac{\zeta}{N\|f\|_{\infty}}\right\} \quad\mbox{and}\quad
\varepsilon=\frac{\xi}{48N2^N}.
\end{equation}
Let 
$$\{\bxset x_1,r_1.,\ldots,\bxset x_M,r_M.\}$$ 
be a finite subfamily such that $\bxset x_i,r_i.\cap K'\neq\emptyset$ for $i=1\ldots,M$ and
\begin{equation}\label{e_k_2}
\mu\left(K'\setminus\bigcup_{i=1}^M\bxset x_i,r_i.\right)<\mu(\Omega)\xi/3.
\end{equation}
For $i=1,\ldots, M$, let $\phi_i\in C^1(\Omega)$ such that $0\leq\phi_i\leq 1$, $\phi_i\equiv 1$ in the set $\bxset x_i,r_i.\setminus \fxset x_i,r_i,\varepsilon.$ , $\phi_i\equiv 0$ outside $\bxset x_i,r_i.$ and

\begin{equation}\label{est_grad}
\|D\phi_i\|_{\infty}\leq \frac{2}{r_i\varepsilon}.
\end{equation}
Denoting 
$$a_i:=\frac{\int_{\bxset x_i,r_i.}f\;d\mu}{\mu(\bxset x_i,r_i.)},$$ we set, for all $x\in\Omega$
$$g(x):=\sum_i\phi_i(x)\langle a_i,x-x_i\rangle.$$
We finally set 
$$K:=\bigcup_{i=1}^M\cl(\bxset x_i,r_i.\setminus \fxset x_i,r_i,\varepsilon.).$$
It is easy to see that $g\in C^1(\Omega)$ and $\|g\|_{\infty}\leq Nr_0\|f\|_{\infty}$, hence property \eqref{e_lem1bis} follows from the choice of $r_0$ in \eqref{e:choice_r_0}. Property \eqref{e_lem1} follows from the inequality
$$\mu(\Omega\setminus \Int(K))\leq \mu(\Omega\setminus K')+\mu(K'\setminus (\bigcup_{i=1}^M\bxset x_i,r_i.))+\mu(( \bigcup_{i=1}^M\bxset x_i,r_i.)\setminus \Int(K))$$
by applying \eqref{e_k_1}, \eqref{e_k_2} and \eqref{eq:neg_frames_cs2}, paired with the choice of $\varepsilon$ in \eqref{e:choice_r_0}. Property \eqref{e_lem2} follows from \eqref{e_unif_cont} by the choice of $r_0$ in \eqref{e:choice_r_0} and that of $a_i$. To prove \eqref{e_lem3}, in the case $p\in[1,\infty)$, we compute, using \eqref{est_grad} and the definition of $g$,
$$\|Dg\|_p^p\leq\sum_i\int_{\bxset x_i,r_i.\setminus\fxset x_i,r_i,\varepsilon.}|a_i|^pd\mu+\int_{\fxset x_i,r_i,\varepsilon.}(2N|a_i|r_i)^p(2/(r_i\varepsilon))^pd\mu.$$
Combining with \eqref{eq:neg_frames_cs2} we have
$$\|Dg\|_p^p\leq\sum_i\mu({\bxset x_i,r_i.})(|a_i|^p+16N2^N\varepsilon^{1-p}(2N|a_i|)^p)$$
and by the definition of $a_i$, this implies
$$\|Dg\|_p^p\leq C\varepsilon^{1-p}\sum_i(\int_{\bxset x_i,r_i.}|f|\;d\mu)^p.$$
Finally, by Jensen's inequality, we get \eqref{e_lem3}. The case $p=\infty$ follows immediately from \eqref{est_grad}.
\end{proof}
\begin{proof}[Proof of Theorem \ref{mainI}]
Suppose that $\varepsilon<1$ and $f$ is not $\mu$-almost everywhere $0$. 

{\bf{First case.}} $f$ is continuous and bounded.
For every $n\geq 1$, set $$\eta_n:=a\varepsilon^22^{-2(n+1)},$$
where 
$$0<a:=\inf_{p\in[1,\infty]}\mu(\Omega)^{-1/p}\|f\|_p.$$ 
We define iteratively a sequence $(\Omega_n,g_n,K_n,f_n)_{n\in\N}$ as follows. Set $\Omega_0:=\Omega, g_0:=0, K_0:=\emptyset, f_0:=f$. Let $n>0$ and assume 
$\Omega_{n-1}, g_{n-1}, K_{n-1}, f_{n-1}$ are given. Apply Lemma \ref{lemma_base}, to obtain compact set $K_n\subset\Omega_{n-1}$ and a function $g_n\in C^1_c(\Omega_{n-1})$ such that
\begin{equation}\label{e_lem1_1}
\mu(\Omega_{n-1}\setminus \Int(K_n))<2^{-n-1}\varepsilon\mu(\Omega_{n-1}),
\end{equation}
\begin{equation}\label{e_lem1_1bis}
\|g_i\|_{\infty}\leq2^{-n}\zeta,
\end{equation}
\begin{equation}\label{e_lem2_2}
|f_{n-1}(x)-Dg_n(x)|\leq\eta_n,\; {\rm{for\;every\;}} x\in K_n.
\end{equation}
\begin{equation}\label{e_lem3_3}
\|Dg_n\|_p\leq C(2^{-n-1}\varepsilon)^{1/p-1}\|f_{n-1}\chi_{\spt(g_n)}\|_p,\;{\rm{for\;every\;}}p\in[1,\infty].
\end{equation}
Finally set $\Omega_n:={\rm{int}}(K_n)$. Define $f_n$ on $\Omega_n$ as
$f_n:=f_{n-1}-Dg_n$. We set $K:=\bigcap_{n>0}K_n$ and
$g:=\sum_{n>0}g_n$. The bound \eqref{e_lem1bis} is an immediate consequence of \eqref{e_lem1_1bis}. We prove now that the set $K$ and the function $g$
satisfy \eqref{e_tmain1}, \eqref{e_tmain2} and \eqref{e_tmain3}. To
prove \eqref{e_tmain1}, notice that by \eqref{e_lem1_1} and
(\ref{eq:neg_frames_cs2}) it holds
$$\mu(\Omega\setminus K)=\sum_{n\geq 0}\mu(\Omega_n\setminus \Int(K_{n+1}))\leq \varepsilon\mu(\Omega).$$
Since, for $n\geq 1$, $\spt({g_{n+1}})\subset{K_n}$, combining \eqref{e_lem2_2} and \eqref{e_lem3_3} with $p=\infty$, we get
$$\|Dg_{n+1}\|_{\infty}\leq C(2^{-(n+1)}\varepsilon)^{-1}\|f_n\chi_{\spt(g_{n+1})}\|_{\infty}\leq C(2^{-(n+1)}\varepsilon)^{-1}\eta_n=C(2^{-(n+1)}\varepsilon)a.$$
This implies that $(\sum_{i=1}^n Dg_i)_{n\in\N}$ (and hence also $(\sum_{i=1}^ng_i)_{n\in\N}$) converges uniformly, therefore $g\in C^1_c(\Omega)$.
Since for $n\geq 1$ it holds
$f=f_{n-1}+\sum_{i=0}^{n-1}Dg_i$ on $K_n$, then \eqref{e_tmain2} follows immediately from \eqref{e_lem2_2}.
To prove \eqref{e_tmain3}, we compute, using \eqref{e_lem2_2} and \eqref{e_lem3_3}
$$\|Dg\|_p\leq\|Dg\chi_K\|_p+\|Dg\chi_{K^c}\|_p\leq\|f\chi_K\|_p+\sum_{n\geq 0}\|Dg\chi_{\Omega_n\setminus\Omega_{n+1}}\|_p$$
$$\leq\|f\|_p+\sum_{n\geq 0}\|Dg_{n+1}\chi_{\Omega_n\setminus\Omega_{n+1}}\|_p\leq\|f\|_p+C2^{n+1}\varepsilon^{1/p-1}\sum_{n\geq 0}\|f_n\chi_{\Omega_n}\|_p$$
$$\leq(1+2C\varepsilon^{1/p-1})\|f\|_p+C2^{n+1}\varepsilon^{1/p-1}\sum_{n\geq 1}\|(f_{n-1}-Dg_n)\chi_{\Omega_n}\|_p$$
$$\leq(1+2C\varepsilon^{1/p-1})\|f\|_p+C2^{n+1}\varepsilon^{1/p-1}\mu(\Omega)^{1/p}\sum_{n\geq 1}\eta_n\leq(1+3C\varepsilon^{1/p-1})\|f\|_p$$
\end{proof}

{\bf{Second case.}} $f$ is Borel.
Fix $\varepsilon>0$. There exists $r>0$ such that 
$$\alpha:=\mu(\{x\in\Omega: |f(x)|>r\})\leq\varepsilon/4.$$ By Lusin's theorem there exists a continuous function $f_1:\Omega\to\R^N$ which agrees with $f$ outside a set of measure $\mu$ less than $\alpha$. The function 
$$f_2(x) =
\bigg \{
\begin{array}{rl}
f_1(x) & {\rm{if }}\; |f_1(x)| \leq r \\
rf_1(x)/|f_1(x)| & {\rm{if }}\; |f_1(x)|> r \\
\end{array}
$$
is continuous and bounded and $\mu(\{x:f(x)\neq f_2(x)\})\leq \varepsilon/2$. Moreover, for every $p\in[1,\infty]$, it holds $\|f_2\|_p\leq 2\|f\|_p$. The theorem follows easily by applying the previous case to the function $f_2$.
%%%%%%%%%%%%%%%%%%%%%%%%%%%%%%%%%%%%%%%%%%%%%%%%%%%%%%%%%%%%%%%%%%%%%%%%%%

%%%%%%%%%%%%%%%%%%%%%%%%%%%%%%%%%%%%%%%%%%%%%%%%%%%%%%%%%%%%%%%%%%%%%%%%%%
\section{Proof of Theorem \ref{main}(II)}\label{s5}
The proof of part (II) of Theorem \ref{main} is quite involved. The reader might find helpful to read Section \ref{s4} before proceeding: although the result in dimension 1 is stronger, the construction presented there requires a considerably smaller amount of technicalities. In the sequel $\glip f.$ denotes the Lipschitz constant of the function $f$.
%%%%%%%%%%%%%%%%%%%%%%%%%%%%%%%%%%%%%%%%%%%%%%%%%%%%%%%%%%%%%%%%%%%%%%%%%%
\subsection{Preliminary results}
\begin{defn}[Local behaviour of a Lipschitz function]
  \label{defn:local-lip-beha}
  Let $S$ be a set and let $\alpha\ge0$ and $r_0>0$. A real-valued
  Lipschitz function $f$ whose domain contains $S$ is said to be
  \textbf{$\alpha$-Lipschitz on $S$ below scale $r_0$} if whenever
  $x,y\in X$ are such that $\dist(x,S),\dist(y,S)\le r_0$
   and $d(x,y)\le r_0$ one has:
  \begin{equation}
    \label{eq:mina1}
    \left| f(x) - f(y) \right| \le\alpha d(x,y).
  \end{equation}
The Lipschitz function $f$ is said to be \textbf{asymptotically flat on $S$} if
for each $\varepsilon>0$ there is an $r_\varepsilon>0$ such that $f$ is
$\varepsilon$-Lipschitz on $S$ below scale $r_\varepsilon$.
\end{defn}
%%%%%%%%%%%%%%%%%%%%%%%%%%%%%%%%%%%%%%%%%%%%%%%%%%%%%%%%%%%%%%%%%%%%%%%%%%
% @def
\pcreatenrm{wfx}{\wform\mu;X.}{|}
\pcreatenrm{wf}{\wform\mu.}{|}
\pcreatenrm{linf}{\infty}{\|}
%%%%%%%%%%%%%%%%%%%%%%%%%%%%%%%%%%%%%%%%%%%%%%%%%%%%%%%%%%%%%%%%%%%%%%%%%%
% \textcolor{red}{To Andrea: I Added some crap on the Weaver
%   differentials and the decomposability bundle. Think if/how to
%   improve it}
\par Before moving on we need to recall something about the general
differentiability theory for real-valued Lipschitz functions in the metric setting
developed in~\cite{deralb} and the differentiability theory,
wrt singular Radon measures, for real-valued Lipschitz functions
defined on Euclidean spaces studied in~\cite{alberti-marchese}. We do
not want to dispirit the reader: both theories can essentially be
treated as black-boxes to understand the results here, as they only
intervene through the Localized Approximation Scheme,
Theorem~\ref{thm:local-appx}. 
\begin{defn}[Alberti representations]
  \label{defn:alb-reps}
  Let $\mu$ be a Radon measure on a metric space $X$ and let
  $\frags(X)$ denote the set of $1$-Lipschitz maps
  $\gamma:\dom\gamma\to X$ where $\dom\gamma$ is a compact subset of
  $\real$. We topologize $\frags(X)$ with the Hausdorff distance
  between graphs. An Alberti representation of $\mu$ is a pair $(Q,w)$
  where $Q$ is a Radon measure on $\frags(X)$ and $w$ is a locally bounded
  Borel map $w:X\to[0,\infty)$ such that:
  \begin{equation}
    \label{eq:alb-reps-1}
    \mu = \int_{\frags(X)}w\gamma_{\#}(\lebmeas 1.\on\dom\gamma)\,dQ(\gamma),
  \end{equation}
where $\gamma_{\#}(\lebmeas 1.\on\dom\gamma)$ denotes the
push-forward, using $\gamma$, of the $1$-dimensional Lebesgue measure
on $\dom\gamma$. More precisely,~(\ref{eq:alb-reps-1}) should be
understood as follows: for each $g:X\to\real$ continuous and in
$L^1(\mu)$ one has:
\begin{equation}
  \label{eq:alb-reps-2}
  \int_Xg\,d\mu = \int_{\frags(X)}dQ(\gamma)\int_{\dom\gamma}w\circ\gamma(t)g\circ\gamma(t)\,dt.
\end{equation}
\end{defn}
\begin{defn}[The norm of the Weaver differential]
  \label{defn:weav-diff}
  Let $\mu$ be a Radon measure on $X$ and $f:X\to\real$ Lipschitz. We
  denote by $\wfnrm df.$ the local norm of $df$~\cite[Defn.~2.101 \&
  2.123]{deralb}, which is an $L^\infty(\mu)$-function which is $\ge0$
  $\mu$-a.e. For this paper we do not need the explicit definition of
  $\wfnrm df.$ but the following
  characterization~\cite[Sec.~3.3]{deralb}: $\wfnrm df.\ge\alpha$ on a
  Borel set $S\subset X$ if and only if for each $\varepsilon>0$ the
  measure $\mu\on S$ has an Alberti representation
  $(Q_\varepsilon,w_\varepsilon)$ such that for
  $Q_\varepsilon$-a.e.~$\gamma$ for $\lebmeas
  1.$-a.e.~$t\in\dom\gamma$ one has
  $(f\circ\gamma)'(t)\ge\alpha-\varepsilon$. From the definition of
  the decomposability bundle in~\cite{alberti-marchese} in terms of
  the Alberti representations we see that if $f:\real^N\to\real$ is
  Lipschitz for $\mu$-a.e.~$x$ one has:
  \begin{equation}
    \label{eq:weav-diff-1}
    \wfnrm df.(x) = \|d_{V(x,\mu)}f\|_2,
  \end{equation}
where $d_{V(x,\mu)}f$ is the derivative of $f$ at $x$ in the direction
of $V(x,\mu)$.
\end{defn}
%%%%%%%%%%%%%%%%%%%%%%%%%%%%%%%%%%%%%%%%%%%%%%%%%%%%%%%%%%%%%%%%%%%%%%%%%%
\begin{thm}[Localized Approximation Scheme]
  \label{thm:local-appx}
  Let $(X,\mu)$ be a locally compact metric measure space ($\mu$ being
  Radon). Let $f$ be a real-valued Lipschitz function defined on $X$
  and $K\subset X$ a compact subset on which $\wfnrm df.\le\alpha$ for
  some $\alpha>0$. Then for each $\varepsilon>0$ there are an
  $r_\varepsilon>0$, a compact $K_\varepsilon\subset K$ and a real-valued
  Lipschitz function $f_\varepsilon$ defined on $X$ such that:
  \begin{description}
  \item[(Apx1)] For any open set $U\subset X$ containing $K$
    the Lipschitz constant
    $\glip f_\varepsilon|U.$ of the restriction $f_\varepsilon|U$ is
    at most the Lipschitz constant $\glip f|U.$ of the restriction
    $f|U$. In particular, taking $U=X$, $\glip f_\varepsilon.\le\glip
    f.$.
  \item[(Apx2)] $\linfnrm f_\varepsilon - f.\le\varepsilon$ and $f$ is
    $\alpha$-Lipschitz on $K_\varepsilon$ below scale $r_\varepsilon$.
  \item[(Apx3)] $\mu(K\setminus K_\varepsilon)\le\varepsilon$.
  \end{description}
\end{thm}
%%%%%%%%%%%%%%%%%%%%%%%%%%%%%%%%%%%%%%%%%%%%%%%%%%%%%%%%%%%%%%%%%%%%%%%%%%
\begin{proof}
  The proof of this result is rather technical and corresponds to
  Theorem 3.66 of~\cite{deralb}, proved in Section 5.1 of~\cite{deralb}, in the special case where $q=1$
  (i.e.~without discussing cones). However, here we need two slight
  modifications of that result: that $X$ is locally compact and
  \textbf{(Apx1)}. We will refer to the notation and proof in Section
  5.1. of~\cite{deralb}.
  \par That in Theorem 3.66 of~\cite{deralb} one can take
  $X$ locally compact is not surprising because in the argument only
  the compactness of $K$ is directly used. 
  \par On the other hand, to obtain \textbf{(Apx1)} we must inspect
  the construction more carefully. We have first constructed a convex
  metric space $Z$ (i.e.~any pair of points is joined by a geodesic)
  and obtained an isometric embedding $i:X\hookrightarrow Z$. Without
  loss of generality we have assumed $\glip f.=1$, considered the
  cylinder $\text{Cyl}=Z\times\real$ (here we use $\real$ instead of a
  finite interval because $X$ is only known to be locally compact)
  with metric:
  \begin{equation}
    \label{eq:local-apx-p1}
    d_{\text{Cyl}}((z_1,t_1),(z_2,t_2)) := \max(|t_1-t_2|, d_Z(z_1,z_2)).
  \end{equation}
  We now identify $X$ with a subset of $\text{Cyl}$ via $x\mapsto
  (i(x),f(x))$. Note that the projection
  \begin{equation}
    \label{eq:local-apx-p2}
    \begin{aligned}
      \tau &:\text{Cyl}\to\real\\
      (z,t) &\mapsto t,
    \end{aligned}
  \end{equation}
extends $f$ as $\tau|X=f$. The goal has then become to approximate
$\tau$, and this has been accomplished by covering $\mu$-a.e. point of $K$
(thus in \textbf{(Apx3)} we pass to a subset $K_\varepsilon$) by
strips whose union is $\mathcal{T}_\varepsilon$ (see the definition of
$\mathcal{T}_n$ above equation (5.41) in~\cite{deralb}). As $K$ is compact
$\mathcal{T}_\varepsilon$ lies in $Z\times[a,b]$ for some $a$, $b$ and
the approximation $\tau_\varepsilon$ is obtained by setting
$\tau_\varepsilon= \tau$ on $Z\times[-\infty,a)$ and
\begin{equation}
  \label{eq:local-apx-p3}
  \tau_\varepsilon(z,t) := a + \int_a^t\chi_{\mathcal{T}^c_\varepsilon}(z,s)\,ds\quad\text{elsewhere.}
\end{equation}
In (5.48) of~\cite{deralb} we have proved that
$\tau_\varepsilon$ is $1$-Lipschitz with respect to the distance:
\begin{equation}
  \label{eq:local-apx-p4}
  D_\alpha((z_1,t_1), (z_2,t_2)) := \max(|t_1-t_2|, \alpha d_Z(z_1,z_2)).
\end{equation}
In particular, as $K\subset U$, $\alpha\le\glip f|U.$. Now pick
$(z_i,t_i)\in U$ for $i\in\{1,2\}$ such that $z_i=i(x_i)$ and
$t_i=f(x_i)$; then:
\begin{equation}
  \label{eq:local-apx-p5}
  D_\alpha((z_1,t_1), (z_2,t_2)) \le \glip f|U. d_X(x_1,x_2),
\end{equation}
which proves the theorem.
\end{proof}
%%%%%%%%%%%%%%%%%%%%%%%%%%%%%%%%%%%%%%%%%%%%%%%%%%%%%%%%%%%%%%%%%%%%%%%%%%
%@def
\def\gdec{^{\normalfont\text{good}}}
\pcreatenrm{linfball}{\infty,B}{\|}
\def\msdec{_{\normalfont\text{m}}} % measure dec
\def\szdec{_{\normalfont\text{s}}} % size dec
\def\spdec{_{\normalfont\text{p}}} % separation dec
%%%%%%%%%%%%%%%%%%%%%%%%%%%%%%%%%%%%%%%%%%%%%%%%%%%%%%%%%%%%%%%%%%%%%%%%%%
\begin{lem}[Step 1 of Construction]
  \label{lem:step-1}
  Let $K\subset\R^N$ be a compact set and assume that the decomposability bundle of 
  $\mu\on K$ has constant dimension $N_0$ and let $\pi_x$ denote its fibre at $x$ and
  $\pi_x^\perp$ its orthogonal complement. Assume that for some
  $N_0$-dimensional hyperplane $\pi$ one has $\linfnrm
  \pi_x-\pi.\le\varepsilon_0$ and let $h:\pi^\perp\to\real$ be
  $1$-Lipschitz. Then there is a constant $C=C(N,N_0)$ (indep.~of $h$)
  such that for each choice of parameters $(\varepsilon\szdec,
  \varepsilon\msdec, \sigma, r_0)\in(0,1/2)^4$ there are a
  $\sqrt{3}$-Lipschitz function $g:\R^N\to\real$, compact subsets
  $J\gdec\subset J\subset K$ and a
  scale $r>0$ such that:  
  \begin{enumerate}[label={\normalfont(\alph*)}]
  \item $\mu(K\setminus J)\le\varepsilon\msdec\mu(K)$,
    $\linfnrm g.\le\varepsilon\szdec$ and $g$ is
    $C\varepsilon_0$-Lipschitz on $J$ below scale
    $r$.
  \item $\mu(J\gdec)\ge
    C^{-1}\sigma^{N-N_0}\mu(J)$.
  \item One can decompose $J\gdec$ as a finite disjoint
    union $\bigcup_{a=1}^{M}C_a$ such that for each $a\in\{1,\cdots,M\}$ there is an
  $0<r_a\le r_0$ such that whenever
    $x\in C_a$ one has:
    \begin{equation}
      \label{eq:mina2}
      \linfballnrm {\tang.g-h}.\le C(\varepsilon\szdec+\sigma),
      \end{equation}
    where in~(\ref{eq:mina2}) $\tang.$ is the map introduced in Definition \ref{def_bu} and we have implicitly extended $h$ as a map
    $h:\real^N=\pi\oplus \pi^\perp\to\real$ by letting $h(y,\tilde
    y)=h(\tilde y)$.
  \end{enumerate}
\end{lem}
%%%%%%%%%%%%%%%%%%%%%%%%%%%%%%%%%%%%%%%%%%%%%%%%%%%%%%%%%%%%%%%%%%%%%%%%%%
Lemma~\ref{lem:step-1} is proven using the following intermediate
results.
%%%%%%%%%%%%%%%%%%%%%%%%%%%%%%%%%%%%%%%%%%%%%%%%%%%%%%%%%%%%%%%%%%%%%%%%%%
% @def
\def\eexset#1,#2.{\text{\normalfont E}\left(
    \setbox1=\hbox{$#1$\unskip}\setbox2=\hbox{$#2$\unskip}
    \ifdim\wd1>0pt #1 \else x\fi,
\ifdim\wd2>0pt#2 \else r\fi\right)
} % rectangle set
\def\ssxset#1,#2.{\text{\normalfont S}\left(
    \setbox1=\hbox{$#1$\unskip}\setbox2=\hbox{$#2$\unskip}
    \ifdim\wd1>0pt #1 \else x\fi,
\ifdim\wd2>0pt#2 \else r\fi\right)
} % core of rectangle set
%%%%%%%%%%%%%%%%%%%%%%%%%%%%%%%%%%%%%%%%%%%%%%%%%%%%%%%%%%%%%%%%%%%%%%%%%%
%%%%%%%%%%%%%%%%%%%%%%%%%%%%%%%%%%%%%%%%%%%%%%%%%%%%%%%%%%%%%%%%%%%%%%%%%%
\def\bdec{_\omega^{\normalfont\text{bad}}}
\def\badset{{\normalfont\text{Bad}}(K,r)}
%%%%%%%%%%%%%%%%%%%%%%%%%%%%%%%%%%%%%%%%%%%%%%%%%%%%%%%%%%%%%%%%%%%%%%%%%%
\begin{lem}[A good rectangle]
  \label{lem:good_rect}
  Let $\mu$ be a Radon measure on $\real^N$ and $r_0>0$. Fix parameters $(L,\sigma)\in
  [8,\infty)\times(0,1/2)$ and define for $0<r\leq r_0$ the following sets:
  \begin{align}
    \label{eq:good_rect_1}
    \eexset ,. &:= x + \left[
                        -\frac{L^2r}{2}, \frac{L^2r}{2}
                      \right]^{N_0} \times
                      \left[
                        -2r,2r
                      \right]^{N-N_0},\\
    \ssxset ,. &:= x + \left[
                        -\frac{L^2r}{2} + \frac{Lr}{2}, \frac{L^2r}{2}
                        - \frac{Lr}{2}
                      \right]^{N_0} \times
                      \left[
                        -2\sigma r,2\sigma r
                      \right]^{N-N_0}.
  \end{align}
Given a compact set $K\subset\R^N$ with $\mu(K)>0$, define the bad set:
      \begin{equation}
    \label{eq:yyja_2}
\begin{split}
    \badset := \biggl\{x\in K:
        0<\mu(\ssxset ,.)\le c\mu(\eexset x,r.)
\biggr\},
\end{split}
\end{equation}
where $c:=\sigma^{N-N_0}2^{-2N-N_0-1}(1+1/6)^{-1}$. 
Then there exists $r<r_0$ (possibly depending on $K$) such that $\mu(\badset)\leq\tfrac{1}{2}\mu(K)$.
\end{lem}
%%%%%%%%%%%%%%%%%%%%%%%%%%%%%%%%%%%%%%%%%%%%%%%%%%%%%%%%%%%%%%%%%%%%%%%%%%
Heuristically, $\eexset ,.$ is a rectangle at $x$ at scale $r$ which
is $L^2$-times bigger in the direction of the first $N_0$ coordinates,
while $\ssxset ,.$ is a \textit{core} of $\eexset ,.$ which is much
smaller (generally $\sigma \ll 1$) in the transverse direction of the
last $N-N_0$ coordinates. Even though $\mu$ is not the Lebesgue
measure, Lemma~\ref{lem:good_rect} says that we can find a \textit{good
  rectangle} $\eexset ,.$ such that the ratio $ \mu(\ssxset ,.)/\mu(\eexset ,.)$
is up to a constant at least the same ratio that one would have for Lebesgue measure.
Besicovitch Covering Theorem implies that:
\begin{cor}[Covering by good rectangles]
  \label{cor:good_rect}
  Let $\mu$ be a Radon measure on $\real^N$ and $K$ compact with
  $\mu(K)>0$ and $r_0>0$. Let $L$, $\sigma$, be as above. Then for
  any $\varepsilon>0$ there are finitely many pairwise disjoint $\{\eexset x_i,r_i.\}_i$
  such that:
  \begin{align}
    \label{eq:good_rect_c1}
          0< r_i&\le r_0,\\
    \label{eq:good_rect_c2}
          \mu(\ssxset x_i,r_i.\cap K) &\ge \sigma^{N-N_0}
    2^{-2N-N_0-1}(1+1/6)^{-1}\mu(\eexset x_i,r_i. \cap K)>0,\\
    \label{eq:good_rect_c3}
          \mu(K\setminus\bigcup_i\eexset x_i,r_i.) &\le\varepsilon\mu(K).
  \end{align}
\end{cor}
\begin{proof}[Proof of Lemma~\ref{lem:good_rect}]
Choose an open set $U\supset K$
    such that
    \begin{equation}
      \label{eq:yyja_1}
      \mu(U)\le\left(1+\frac{1}{6}\right)\mu(K).
    \end{equation}
    Then choose $r\le r_0$ such that $\eexset ,4r.\cap K\ne\emptyset$
    implies that $\eexset ,4r.\subset U$. 
Then let $I$ denote the integral:
\begin{equation}
  \label{eq:yyja_3}
  I := \int\chi_{\badset}(x_1)\chi_{\ssxset
    x_1,r/2.}(x_2)\,d\mu(x_1)d\lebmeas N.(x_2);
\end{equation}
if we integrate first in $x_2$ we get:
\begin{equation}
  \label{eq:yyja_4}
  I = r^N\sigma^{N-N_0}(L^2-L)^{N_0}2^{N-2N_0}\mu(\badset).
\end{equation}
If we integrate first in $x_1$ we get:
\begin{equation}
  \label{eq:yyja_5}
  I = \int \mu\left(\badset \cap\ssxset x_2,r/2.\right)\,d\lebmeas N.(x_2).
\end{equation}
Choose $x_2$ such that $\badset \cap\ssxset x_2,r/2.\ne\emptyset$ and
let $x_1\in K$ denote a point in this non-empty intersection. Then for
$i\in\{1,\cdots, N_0\}$ we get:
\begin{equation}
  \label{eq:yyja_6}
  |x_1^i-x_2^i|\le\left(\frac{L^2}{2}-\frac{L}{2}\right)r;
\end{equation}
for $i>N_0$ one has $|x_1^i-x_2^i|\le\sigma r$. In particular,
$\badset \cap\ssxset x_2,r/2.\subset\ssxset x_1,r.$; but as
$x_1\in\badset$ one has $\mu(\ssxset x_1,r.)\le c\mu(\eexset
x_1,r.)$. Using the triangle inequality we observe $\eexset
x_1,r.\subset \eexset x_2,2r.$ from which we conclude $\eexset
x_2,2r.\subset\eexset x_1,4r.\subset U$. We thus obtain the upper
bound:
\begin{equation}
  \label{eq:yyja_7}
  \begin{split}
    I& \le c\int \mu\left(U\cap\eexset x_2,2r.\right)\,d\lebmeas
    N.(x_2)\\
    &=c\int\chi_U(x_1)\chi_{\eexset x_2,2r.}(x_1)\,d\lebmeas
    N.(x_2)d\mu(x_1)\\
    &=c\int\chi_U(x_1)\chi_{\eexset x_1,2r.}(x_2)\,d\lebmeas
    N.(x_2)d\mu(x_1)\\
    &=cr^N(L^2)^{N_0}2^{3N-2N_0}\mu(U)\le cr^N(L^2)^{N_0}2^{3N-2N_0}\left(1+\frac{1}{6}\right)\mu(K).
  \end{split}
\end{equation}
The proof is completed combining~(\ref{eq:yyja_7})
with~(\ref{eq:yyja_4}) and the choice of $c$ which gives $\mu(\badset)\leq\tfrac{1}{2}\mu(K)$.

\end{proof}
%%%%%%%%%%%%%%%%%%%%%%%%%%%%%%%%%%%%%%%%%%%%%%%%%%%%%%%%%%%%%%%%%%%%%%%%%%
\begin{proof}[Proof of Lemma~\ref{lem:step-1}]
  \makestep{1}{Construction of auxiliary functions}
  \par Without loss of generality we will assume that $\pi_0$ is the
  plane $\real^{N_0}\times\{0\}$.
  \par Recall that the $1$-Lipschitz retraction of $\real^{N_0}$ onto
  $B\subset\real^{N_0}$ is given by:
  \begin{equation}
    \label{eq:jjss_1}
    J(x) :=
    \begin{cases}
      x&\text{if $|x|\le1$}\\
      \frac{x}{|x|}&\text{otherwise}.
    \end{cases}
  \end{equation}
  Fix the parameter $L\gg 1$; we define a $\frac{4}{L}$-Lipschitz
  cut-off function on $\real$:
  \begin{equation}
    \label{eq:jjss_2}
    \varphi(r) := 
    \begin{cases}
      1&\text{if $|r|\in[0, L^2/2-L/4]$,}\\
      1 - \frac{4}{L}(|r|-L^2/2+L/4)&\text{if $|r|\in(L^2/2-L/4,
        L^2/2]$,}\\
      0&\text{otherwise.}
    \end{cases}
  \end{equation}
  We also define the $1$-Lipschitz cut-off function on $\real$:
  \begin{equation}
    \label{eq:jjss_3}
    \psi(r) := 
    \begin{cases}
      1&\text{if $|r|\le1$,}\\
      2-|r|&\text{if $|r|\in(1,2]$,}\\
      0&\text{otherwise.}
    \end{cases}
  \end{equation}
  We now replace $h$ by $h\circ J$ so that we can assume $\|h\|_\infty\le
  1$ and ${\partial_r}h=0$ on $\overline{B}^c$.
  \par We define the building block of our construction:
  \begin{equation}
    \label{eq:jjss_4}
    F(y,\tilde y) := \varphi(|y|)\psi(|\tilde y|)h(\tilde y).
  \end{equation}
  We now collect some properties of $F$:
  \begin{description}
  \item[(F1)] $F$ is $(4/L+\sqrt{2})$-Lipschitz (note that on
    $\overline{B}^c$ $\nabla\psi$ and $\nabla h$ give orthogonal
    contributions to the derivative of $F$).
  \item[(F2)] $F = 0$ outside of $B = \left[-\frac{L^2}{2},
      \frac{L^2}{2}\right]^{N_0} \times\left[-2, 2\right]^{N-N_0}$.
  \item[(F3)] $F=h$ on the \textit{core} $S = \left[-\frac{L^2}{2} +
      \frac{L}{4}, \frac{L^2}{2} -
      \frac{L}{4}\right]^{N_0}\times[-1,1]^{N-N_0}$.
  \item[(F4)] $\|F\|_\infty\le \|h\|_\infty\le 1$.
  \item[(F5)] We can assume $\wfnrm df.\le C\varepsilon_0$,
    where $C$ depends possibly only on $N$ and $N_0$.
  \item[(F6)] $F$ decays linearly to $0$ when approaching the boundary of $B$:
    $|F(y,\tilde y)|\le d((y,\tilde y),\partial B)$.
  \end{description}
Only \textbf{(F5)} and \textbf{(F6)} require justification. For
\textbf{(F5)} observe that by the Leibniz rule $df = \psi hd\varphi + \varphi h d\psi
    + \varphi\psi dh$; observe also that $d\psi$ and $dh$ give an $O(\varepsilon_0)$
    contribution to the Weaver differential (we take $\mu$ as the
    reference measure) as
    $\|\pi_x-\pi_0\|_\infty\le\varepsilon_0$. Thus, $\wfnrm
    df.\le\frac{4}{L}+C\varepsilon_0$, and, choosing $L$ large enough
    and inflating $C$, we get \textbf{(F5)}. For \textbf{(F6)} we have
    three cases. The first: $|y|\ge L^2/2-L/4$ so that:
    \begin{equation}
      \label{eq:jjss_5}
      \begin{split}
        |F(y,\tilde y)|&\le|\varphi(y)|\le
        1-\frac{4}{L}\left(|y|-\frac{L^2}{2}+\frac{L}{4}\right)\\
        &=\frac{4}{L}\left(\frac{L^2}{2}-|y|\right)\\
        &\le\frac{4}{L}d((y,\tilde y),\partial B),
      \end{split}
    \end{equation}
and \textbf{(F6)} holds as long as $L\ge 4$. The second: $|y|\le
L^2/2-L/4$ and $|\tilde y|\ge 1$:
\begin{equation}
  \label{eq:jjss_6}
  |F(y,\tilde y)|\le\psi(\tilde y)=2-|\tilde y|\le d((y,\tilde y),\partial B). 
\end{equation}
The third: $|y|\le
L^2/2-L/4$ and $| \tilde y|\le 1$: $d((y,\tilde y),\partial B)\ge1$ and
so we conclude by \textbf{(F4)}.
\makestep{2}{Covering $K$ by good boxes}
\par We apply Corollary~\ref{cor:good_rect} and find finitely many
pairwise disjoint
$\{\eexset x_i,r_i.\}_{i=1}^M$ such that:
  \begin{align}
    \label{eq:jjss_7}
          0< r_i&\le \min(r_0, \varepsilon\szdec/2),\\
    \label{eq:jjss_8}
          \mu(\ssxset x_i,r_i.\cap K) &\ge C^{-1}\sigma^{N-N_0}
          \mu(\eexset x_i,r_i. \cap K)>0,\\
    \label{eq:jjss_9}
          \mu(K\setminus\bigcup_i\eexset x_i,r_i.) &\le\frac{\varepsilon\msdec}{2}\mu(K).
  \end{align}
We let $\hat J := K\cap\bigcup_i\eexset x_i,r_i.$ and $\hat J\gdec :=
K\cap\bigcup_i\eexset x_i,r_i.$ which give (b) and the first
inequality in (a) if we replace $J$ with $\hat J$: the set $J$ will be
chosen later to be a subset of $\hat J$ and $J\gdec$ will be set to be
$J\cap\hat J\gdec$.
\par Write $x_i=(y_i,\tilde y_i)$ and define the
$(4/L+\sqrt{2})$-Lipschitz function $F_i$ supported on $\eexset
x_i,r_i.$:
\begin{equation}
  \label{eq:jjss_10}
  F_i(y,\tilde y) = \sigma r_i F\left(\frac{y-y_i}{r_i}, \frac{\tilde
      y -\tilde y_i}{\sigma r_i}\right).
\end{equation}
Because of \textbf{(F6)} the function $F_i$ can be glued together to
get a $(4/L+\sqrt{2})$-Lipschitz function $f$ as in~\cite[Thm.~4.8]{dim_blow}. Note that the choice of
the $r_i$'s implies $\|f\|_\infty\le\varepsilon\szdec/2$.
\par If $x\in\ssxset x_i,r_i.$ lies on the core center, i.e.~$x$ is of
the form $x=(y,\tilde y_i)$, then:
    \begin{equation}
      \label{eq:jjss_11}
      \linfballnrm {\tang x, \sigma r_i.f-h}.=0.
    \end{equation}
Thus, as $f$ is $(4/L+\sqrt{2})$-Lipschitz, for all $x\in\ssxset
x_i,r_i.$ we have:
    \begin{equation}
      \label{eq:jjss_12}
      \linfballnrm {\tang x, \sigma r_i.f-h}.\le C\sigma.
    \end{equation}
\makestep{3}{Applying the approximation scheme}
\par Note that \textbf{(F5)} implies that $\wfnrm df.\le
C\varepsilon_0$ and applying Theorem~\ref{thm:local-appx} we can find
a $(4/L+\sqrt{2})$-Lipschitz function $g$, a compact set $J\subset\hat
J$ and $r>0$ such that:
\begin{align}
  \label{eq:jjss_13}
  \|g-f\|_\infty &\le\frac{\varepsilon\szdec}{2}\min\left(\frac{1}{2},
    \min_{1\le i\le M}(\sigma r_i)\right),\\
  \label{eq:jjss_14}
  \mu(\hat J\setminus J)&\le\frac{\varepsilon\msdec}{2}\mu(K),\\
  \label{eq:jjss_15}
  \mu(J\cap\hat
  J\gdec)&\le\frac{C^{-1}\varepsilon\msdec}{16}\sigma^{N-N_0}\mu(\hat J),
\end{align}
and $g$ is $(C\varepsilon_0)$-Lipschitz on $J$ below scale $r$. Thus,
if we let $J\gdec := \hat J\gdec \cap J$, then (a) and (b) follow. For
(c) we just combine~(\ref{eq:jjss_12}) with:
\begin{equation}
  \label{eq:jjss_16}
  \frac{\|f-g\|_\infty}{\sigma r_i}\le \frac{\varepsilon\szdec}{2}.
\end{equation}
We finally choose $L$ large enough so that \textbf{(F6)} holds and
$\frac{4}{L}\le \varepsilon\szdec$.
\end{proof}
%%%%%%%%%%%%%%%%%%%%%%%%%%%%%%%%%%%%%%%%%%%%%%%%%%%%%%%%%%%%%%%%%%%%%%%%%%
\begin{lem}[Step 2 of Construction]
  \label{lem:step-2}
Let $K\subset\real^N$ be a compact subset and $\alpha>0$ be such that
$\wfnrm df.\le\alpha$ on $K$. Then there is a constant $C=C(N)$
(indep.~of $f$) such that for any choice of parameters
$(\varepsilon\szdec, \varepsilon\msdec)\in(0,1/2)^2$ there are a
Lipschitz function $\hat f$ and a compact set $\hat K$ such that:
\begin{enumerate}[label={\normalfont(\alph*)}]
\item $\linfnrm \hat f-f.\le\varepsilon\szdec$ and $\mu(K\setminus
  \hat K)\le\varepsilon\msdec\mu(K)$.
\item $\hat f$ is asymptotically flat on $\hat K$.
\item $\hat f$ is $(\glip f. +C\frac{\alpha}{\varepsilon\msdec})$-Lipschitz.
\end{enumerate}
\end{lem}
%%%%%%%%%%%%%%%%%%%%%%%%%%%%%%%%%%%%%%%%%%%%%%%%%%%%%%%%%%%%%%%%%%%%%%%%%%
% @ def
\def\onedec{^{(1)}}
\def\twodec{^{(2)}}
\def\jjdec{^{(j)}}
\def\jonedec{^{(j+1)}}
\def\lldec{^{(l)}}
\def\lonedec{^{(l+1)}}
\def\kkdec{^{(k)}}
\def\konedec{^{(k+1)}}
%%%%%%%%%%%%%%%%%%%%%%%%%%%%%%%%%%%%%%%%%%%%%%%%%%%%%%%%%%%%%%%%%%%%%%%%%%
\begin{proof}
  \makestep{1}{Killing the gradient of $f$}
  \par We apply Theorem~\ref{mainI}
  with parameters $\varepsilon:=\varepsilon\msdec\onedec>0$,
  $\zeta:=\varepsilon\szdec\onedec>0$ to find a
  $\frac{C\alpha}{\varepsilon\onedec\msdec}$-Lipschitz function $f_1$
  and a compact $K_1\subset K$ such that
  $\|f_1\|_\infty\le\varepsilon\szdec\onedec$, $df_1 = df$ on $K_1$
  and 
  \begin{equation}
    \label{eq:uujj_1}
    \mu(K\setminus K_1)\le\varepsilon\msdec\onedec\mu(K).
  \end{equation}
  We let $g_1:=f-f_1$ and observe that $f_1$ is $(\glip
  f.+\frac{C\alpha}{\varepsilon\onedec\msdec})$-Lipschitz with
  $dg_1=0$ on $K_1$. We fix the parameters
  $(\alpha_1,\eta\szdec\onedec,\eta\msdec\onedec)\in(0,1)^3$ and use
  Theorem~\ref{thm:local-appx} to find an $(\glip
  f.+\frac{C\alpha}{\varepsilon\onedec\msdec})$-Lipschitz function
  $\hat g_1$ and a compact $H_1\subset K_1$ and a scale $r_1>0$ such that:
  \begin{align}
    \label{eq:uujj_2}
    \|\hat g_1 - g_1\|_\infty&\le\eta\szdec\onedec\\
    \label{eq:uujj_3}
    \mu(K_1\setminus H_1)&\le\eta\msdec\onedec\mu(K_1),
  \end{align}
and $\hat g_1$ is $\alpha_1$-Lipschitz on $H_1$ below scale $r_1$.
\makestep{2}{The general iteration}
\par We apply Theorem~\ref{mainI}
  with parameters $\varepsilon:=\varepsilon\msdec\jonedec>0$,
  $\zeta:=\varepsilon\szdec\jonedec>0$ to find a
  $\frac{C\alpha_j}{\varepsilon\msdec\jonedec}$-Lipschitz function
  $f_{j+1}$ and a compact $K_{j+1}\subset H_j$ such that:
  \begin{align}
    \label{eq:uujj_4}
    \|f_{j+1}\|_\infty&\le\varepsilon\szdec\jonedec\\
    \label{eq:uujj_5}
    df_{j+1}&=d\hat g_j\quad\text{on $K_{j+1}$}\\
    \label{eq:uujj_6}
    \mu(H_j\setminus K_{j+1})&\le\varepsilon\msdec\jonedec\mu(H_j).
  \end{align}
  We let $g_{j+1}:=\hat g_j-f_{j+1}$ and observe that $g_{j+1}$
  satisfies:
  \begin{equation}
    \label{eq:uujj_7}
    \glip g_{j+1}.\le \glip f. +
    \frac{C\alpha}{\varepsilon\msdec\onedec} +\sum_{l\le j}
    \frac{C\alpha_l}{\varepsilon\msdec\lonedec},
  \end{equation}
  and satisfies $dg_{j+1}$ on $K_{j+1}$. Moreover, by the inductive
  step and because of \textbf{(Apx1)} in Theorem~\ref{thm:local-appx}
  we can assume that for $l\le j$ the function $g_{j+1}$ is
  $(\alpha_l +
  \sum_{k=l}^j\frac{C\alpha_k}{\varepsilon\msdec\konedec})$-Lipschitz on $H_l$
  below scale $r_l$.
  \par We fix the parameters
  $(\alpha_{j+1},\eta\szdec\jonedec,\eta\msdec\jonedec)\in(0,1)^3$ and use
  Theorem~\ref{thm:local-appx} to find an $(\glip
  f.+\frac{C\alpha}{\varepsilon\onedec\msdec}+\sum_{l\le j}
    \frac{C\alpha_l}{\varepsilon\msdec\lonedec})$-Lipschitz function
  $\hat g_{j+1}$ and a compact $H_{j+1}\subset K_{j+1}$ and a scale $r_{j+1}\in(0,r_j)$ such that:
  \begin{align}
    \label{eq:uujj_8}
    \|\hat g_{j+1} - g_{j+1}\|_\infty&\le\eta\szdec\jonedec\\
    \label{eq:uujj_9}
    \mu(K_{j+1}\setminus H_{j+1})&\le\eta\msdec\jonedec\mu(K_{j+1}),
  \end{align}
and $\hat g_{j+1}$ is $\alpha_{j+1}$-Lipschitz on $H_{j+1}$ below
scale $r_{j+1}$. Also by \textbf{(Apx1)} in
Theorem~\ref{thm:local-appx} we see that $\hat g_{j+1}$ is $(\alpha_l +
  \sum_{k=l}^j\frac{C\alpha_k}{\varepsilon\msdec\konedec})$-Lipschitz on $H_l$
  below scale $r_l$.
  \makestep{3}{Choice of parameters}
  \par The parameters $\alpha_j$, $\varepsilon\szdec\jjdec$,
  $\eta\szdec\jjdec$ and $\eta\msdec\jjdec$ can be chosen arbitrarily
  small at each stage, while with the parameters
  $\varepsilon\msdec\jjdec$ we must be careful otherwise the Lipschitz
  constants of the functions involved at each stage will diverge to
  $\infty$. For $j\ge 1$ we thus let:
  \begin{equation}
    \label{eq:uujj_10}
    \varepsilon\msdec\jonedec:=\sqrt{\alpha_j}.
  \end{equation}
  We have the bound:
  \begin{equation}
    \label{eq:uujj_11}
    \begin{split}
      \mu(K\setminus H_j)&\le\mu(K\setminus K_1)+\sum_{l\le
        j}\mu(K_l\setminus H_l) + \sum_{k=1}^{j-1}\mu(H_l\setminus
      K_{l+1})\\
      &\le\varepsilon\msdec\onedec\mu(K) + \sum_{l\le
        j}\eta\msdec\lldec\mu(K_l) +
      \sum_{k=1}^{j-1}\varepsilon\msdec\lonedec\mu(H_l)\\
      &\le\left(\varepsilon\msdec\onedec + \sum_{l\le
        j}\eta\msdec\lldec + \sum_{k=1}^{j-1}\sqrt{\alpha_l}\right)\mu(K).
    \end{split}
  \end{equation}
  We let $H_\infty = \bigcap_jH_j$ and want this set to contain a
  significant part of the measure of $K$. Thus, if we choose
  $\varepsilon\msdec\onedec := \varepsilon\msdec/3$ and the
  $\eta\msdec\lldec$, $\alpha_l$ so that:
  \begin{align}
    \label{eq:uujj_12}
    \sum_l\eta\msdec\lldec&\le\frac{\varepsilon\msdec}{3}\\
    \label{eq:uujj_13}
    \sum_l\sqrt{\alpha_l}&\le\frac{\varepsilon\msdec}{3},
  \end{align}
  we get $\mu(K\setminus H_\infty)\le\varepsilon\msdec\mu(K)$ and can
  finally let $\hat K := H_\infty$.
  \par We now want to guarantee convergence of the sequence $\{\hat
  g_j\}$. First note:
  \begin{equation}
    \label{eq:uujj_14}
    \hat g_j - f = \sum_{l\le j}(\hat g_l - g_l) - \sum_{l\le j}f_l,
  \end{equation}
  from which we deduce:
  \begin{equation}
    \label{eq:uujj_15}
    \|\hat g_j - f\|_\infty\le\sum_{l\le j}(\eta\szdec\lldec+\varepsilon\szdec\lldec);
  \end{equation}
  we will choose the parameters $\eta\szdec\lldec$,
  $\varepsilon\szdec\lldec$ so that:
  \begin{equation}
    \label{eq:uujj_16}
    \sum_{l}(\eta\szdec\lldec+\varepsilon\szdec\lldec)\le\varepsilon\szdec.
  \end{equation}
  Now $\hat g_{j+1}-\hat g_j=\hat g_{j+1} - g_{j+1}-f_j$ and thus:
  \begin{equation}
    \label{eq:uujj_17}
    \|\hat g_{j+1}-\hat g_j\|_\infty\le \eta\szdec\jonedec+\varepsilon\szdec\jjdec.
  \end{equation}
  Therefore for $j\to\infty$ we have $\hat g_j\to \hat f$ uniformly,
  $\hat f$ being a continuous function; but:
  \begin{equation}
    \label{eq:uujj_18}
    \glip\hat g_j.\le \glip
    f.+\frac{3C\alpha}{\varepsilon\msdec}+C\sum_{l\le j}\sqrt{\alpha_j},
  \end{equation}
 and if we choose the $\alpha_j$ to satisfy:
 \begin{equation}
   \label{eq:uujj_19}
   \sum_j\sqrt{\alpha_j}\le\frac{\alpha}{\varepsilon\msdec}
 \end{equation}
 and then inflate $C$ we conclude that $\hat f$ is $(\glip f. +
 \frac{C\alpha}{\varepsilon\msdec})$-Lipschitz. Finally for $l\le j$
 the function $\hat g_j$ is $(\alpha_l+C\sum_{l\le k\le
   j}\sqrt{\alpha_k})$-Lipschitz on $H_l\supset H_\infty=\hat K$ below
 scale $r_l$, and thus $\hat f$ is asymptotically flat on $\hat K$.
\end{proof}
%%%%%%%%%%%%%%%%%%%%%%%%%%%%%%%%%%%%%%%%%%%%%%%%%%%%%%%%%%%%%%%%%%%%%%%%%%
\begin{lem}[Step 3 of Construction]
  \label{lem:step-3}
  Let $K\subset\real^N$ be compact and assume that the decomposability
  bundle of $\mu\on K$ has constant dimension $N_0$. Assume also that
  for some $N_0$-dimensional hyperplane $\pi$ one has $\linfnrm
  \pi_x-\pi.\le\varepsilon_0$ for each $x\in K$ where
  $\varepsilon_0>0$. Let $h:\pi^{\perp}\to\real^N$ be
  $1$-Lipschitz with $h(0)=0$ and extend it to $\real^N$ as in
  Lemma~\ref{lem:step-1}. Then there are constants $C_0,C_1$, which depend only
  on $N$ and $N_0$, such that the following holds: for each choice of parameters
  $(\varepsilon\szdec,\varepsilon\msdec,r_0)\in(0,1/2)^3$ there are a
  $(\sqrt{3}+C_0\frac{\varepsilon_0}{\varepsilon\msdec^2})$-Lipschitz function 
  $g:\real^N\to\real$ and a compact $J\subset
  K$ such that:
  \begin{enumerate}[label={\normalfont(\alph*)}]
  \item $g$ is asymptotically flat on $J$,
    $\linfnrm g.\le C_1\varepsilon\szdec$ and $\mu(K\setminus
    J)\le C_1\varepsilon\msdec\mu(K)$.
  \item One can write $J$ as a finite disjoint union
    $J = \bigcup_{a=1}^{M}C_a$ and
    for each $a\in{1,\cdots, M}$ there is an $0<r_a\le r_0$ such that if $x\in
    C_a$ one has:
    \begin{equation}
      \label{eq:mina3}
      \linfballnrm {\tang.g-h}.\le C_1(\varepsilon_0 +\varepsilon\msdec^{1/(N-N_0)}).
    \end{equation}
  \end{enumerate}
\end{lem}
%%%%%%%%%%%%%%%%%%%%%%%%%%%%%%%%%%%%%%%%%%%%%%%%%%%%%%%%%%%%%%%%%%%%%%%%%%
\def\onedec{^{(1)}}
\def\twodec{^{(2)}}
\def\jjdec{^{(j)}}
\def\lldec{^{(l)}}
\def\kkdec{^{(k)}}
\def\konehalfdec{_{k\frac{1}{2}}}
\def\jonedec{^{(j+1)}}
\def\konedec{^{(k+1)}}
\def\onehalfdec{_{1\frac{1}{2}}}
\def\jonehalfdec{_{j\frac{1}{2}}}
\begin{proof}
  \makestep{1}{Applying Lemma~\ref{lem:step-1}}
  \par We let $C$ denote the maximum of the constants $C$, $C_0$ and
  $C_1$ from Lemmas~\ref{lem:step-1} and \ref{lem:step-2}. We fix
  parameters $(\varepsilon\szdec\onedec, \varepsilon\msdec\onedec,
  \sigma, r\onedec)\in(0,1)^4$ to be chosen later. For the moment we
  just remark we will need $r\onedec \le r_0$ and
  $\frac{C^{-1}\sigma^{N-N_0}}{2}>\varepsilon\msdec\onedec$. We now
  apply Lemma~\ref{lem:step-1} to obtain an $\sqrt{3}$-Lipschitz function
  $g_1$ and compact subsets $J_1\gdec\subset J_1\subset K$ and a scale $0<\rho_1\le
  r\onedec$ such that:
  \begin{description}
  \item[($g_1$:a)] $\mu(K\setminus J_1)\le\varepsilon\msdec\onedec\mu(K)$,
    $\linfnrm g_1.\le\varepsilon\szdec\onedec$ and $g_1$ is
    $C\varepsilon_0$-Lipschitz on $J_1$ below scale
    $\rho_1$.
  \item[($g_1$:b)] $\mu(J_1\gdec)\ge
    C^{-1}\sigma^{N-N_0}\mu(J_1)$.
  \item[($g_1$:c)] One can decompose $J_1\gdec$ as a finite disjoint
    union $\bigcup_{a=1}^{M_1}C_a\onedec$ such that for each $a\in\{1,\cdots,M_1\}$ there is an
  $0<r_a\onedec\le r\onedec$ such that whenever
    $x\in C_a\onedec$ one has:
    \begin{equation}
      \label{eq:moto1}
      \linfballnrm {\tang x,r_a\onedec.g_1-h}.\le C(\varepsilon\szdec\onedec+\sigma).
      \end{equation}
    \end{description}
   \makestep{2}{Applying Lemma~\ref{lem:step-2}} 
   \par We fix
  parameters $(\hat\varepsilon\szdec\onedec,
  \hat\varepsilon\msdec\onedec)\in(0,1)^2$ to be chosen later. We
  apply Lemma~\ref{lem:step-2} (to the function $g_1$ and the compact
  set $J_1$) to find a Lipschitz function $\hat g_1$ and a compact set
  $\hat J_1\subset J_1$ such that:
  \begin{description}
  \item[($\hat g_1$:a)] $\linfnrm \hat g_1- g_1.\le\hat\varepsilon\onedec\szdec$ and $\mu(J_1\setminus
  \hat J_1)\le\hat\varepsilon\msdec\onedec\mu(J_1)$.
\item[($\hat g_1$:b)] $\hat g_1$ is asymptotically flat on $\hat J_1$.
\item[($\hat g_1$:c)] $\hat g_1$ is $(\sqrt{3}
 +C^2\frac{\varepsilon_0}{\hat\varepsilon\msdec\onedec})$-Lipschitz.
\end{description}
As in \textbf{($g_1$:c)} pick $x\in C_a\onedec$ and
combine~(\ref{eq:moto1}) with \textbf{($\hat g_1$:a)} to obtain:
\begin{equation}
  \label{eq:moto2}
  \begin{split}
  \linfballnrm {\tang x,r_a\onedec.\hat g_1-h}. &\le \linfballnrm
  {\tang x,r_a\onedec.\hat g_1- \tang x,r_a\onedec. g_1}. + 
  C(\varepsilon\szdec\onedec+\sigma)\\
  &\le\frac{2\hat\varepsilon\szdec\onedec}{r_a\onedec} +  C(\varepsilon\szdec\onedec+\sigma),
\end{split}
\end{equation}
and observe that $\hat\varepsilon\szdec\onedec$ will be chosen later
to be insignificant next to $\min_{1\le a\le M_1}{r_a\onedec}$.
\makestep{3}{The construction of $g_2$, $\hat g_2$ and $G_2$}
\par We now want to follow the first two steps, but first need an
intermediate construction. Fix parameters $(\alpha_1,
\varepsilon\onedec\spdec)\in(0,1)^2$ and choose
$R_1\le\alpha_1^2\min_{1\le a\le M_1}r_a\onedec$ such that $\hat g_1$
is $\alpha_1$-Lipschitz on $\hat J_1$ below scale $R_1$. Find a
compact $J\onehalfdec\subset\hat J_1\setminus J_1\gdec$ such that:
\begin{equation}
  \label{eq:moto3}
  \mu\left(
    (\hat J_1\setminus J_1\gdec) \setminus J\onehalfdec
\right) \le \varepsilon\spdec\onedec \mu(\hat J_1\setminus J_1\gdec),
\end{equation}
and find a $\tau_1>0$ such that the $\tau_1$-neighbourhoods of
$J\onehalfdec$ and $J_1\gdec$ are disjoint.
\par We now fix parameters $(\varepsilon\szdec\twodec, \varepsilon\msdec\twodec,
  r\twodec)\in(0,1)^3$ (for
  the moment imposing the constraint that $2r\twodec < \tau_1$)
 and apply Lemma~\ref{lem:step-1} using the parameters
$(\varepsilon\szdec\twodec, \varepsilon\msdec\twodec,
  \sigma, r\twodec)$ and the compact set $J\onehalfdec$ to obtain an $\sqrt{3}$-Lipschitz function
  $g_2$ and compact subsets $J_2\gdec\subset J_2\subset J\onehalfdec$ and a scale $0<\rho_2\le
  r\twodec$ such that:
  \begin{description}
  \item[($g_2$:a)] $\mu(J\onehalfdec\setminus J_2)\le\varepsilon\msdec\twodec\mu(J\onehalfdec)$,
    $\linfnrm g_2.\le\varepsilon\szdec\twodec$ and $g_2$ is
    $C\varepsilon_0$-Lipschitz on $J_2$ below scale
    $\rho_2$.
  \item[($g_2$:b)] $\mu(J_2\gdec)\ge
    C^{-1}\sigma^{N-N_0}\mu(J_2)$.
  \item[($g_2$:c)] One can decompose $J_2\gdec$ as a finite disjoint
    union $\bigcup_{a=1}^{M_2}C_a\twodec$ such that for each $a\in\{1,\cdots,M_2\}$ there is an
  $0<r_a\twodec\le r\twodec$ such that whenever
    $x\in C_a\twodec$ one has:
    \begin{equation}
      \label{eq:moto4}
      \linfballnrm {\tang x,r_a\twodec.g_2-h}.\le C(\varepsilon\szdec\twodec+\sigma).
      \end{equation}
    \end{description}
   \par We fix
  parameters $(\hat\varepsilon\szdec\twodec,
  \hat\varepsilon\msdec\twodec)\in(0,1)^2$ to be chosen later. We
  apply Lemma~\ref{lem:step-2} (to the function $g_2$ and the compact
  set $J_2$) to find a Lipschitz function $\hat g_2$ and a compact set
  $\hat J_2\subset J_2$ such that:
  \begin{description}
  \item[($\hat g_2$:a)] $\linfnrm \hat g_2 - g_2.\le\hat\varepsilon\twodec\szdec$ and $\mu(J_2\setminus
  \hat J_2)\le\hat\varepsilon\msdec\twodec\mu(J_2)$.
\item[($\hat g_2$:b)] $\hat g_2$ is asymptotically flat on $\hat J_2$.
\item[($\hat g_2$:c)] $\hat g_2$ is $(\sqrt{3} +C^2\frac{\varepsilon_0}{\hat\varepsilon\msdec\twodec})$-Lipschitz.
\end{description}
\par We now modify $\hat g_2$ so that it vanishes on the
$(\tau_1)$-neighbourhood of $J_1\gdec$ and stays the same on the
$(\tau_1/2)$-neighbourhood of $J\onehalfdec$. This is accomplished by
replacing $\hat g_2$ with
\begin{equation}
  \label{eq:moto5}
  \hat g_2\max\left( 0, 
    1 - \frac{2}{\tau_1}\dist(\cdot, (J\onehalfdec)_{\tau_1/2})
\right),
\end{equation} where $(J\onehalfdec)_{\tau_1/2}$ denotes the
$(\tau_1/2)$-neighbourhood of $J\onehalfdec$. As $(J_1\gdec)_{\tau_1}$
and $(J\onehalfdec)_{\tau_1}$ are disjoint, $\hat g_2$ now vanishes on
$(J_1\gdec)_{\tau_1}$. One also obtains an upper bound on the
Lipschitz constant of $\hat g_2$:
\begin{equation}
  \label{eq:moto6}
  \glip\hat g_2. \le \sqrt{3} +
  C^2\frac{\varepsilon_0}{\hat\varepsilon\twodec\msdec} 
  + 2\frac{\varepsilon\szdec\twodec + \hat\varepsilon\szdec\twodec}{\tau_1}.
\end{equation}
To get the last term in~(\ref{eq:moto6}) to be $\le\alpha_1$ we impose
the restriction $2(\varepsilon\szdec\twodec +
\hat\varepsilon\szdec\twodec)\le\alpha_1\tau_1$.
\par We now let $G_2=\hat g_1 + \hat g_2$ and to get a good upper
bound on $\glip G_2.$ impose the restriction $2(\varepsilon\szdec\twodec +
\hat\varepsilon\szdec\twodec)\le\alpha_1 R_1$. In fact, we now verify
that 
\begin{equation}
  \label{eq:moto7}
  \glip G_2. \le  \sqrt{3} +
  C^2\frac{\varepsilon_0}{\min(\hat\varepsilon\onedec\msdec,
\hat\varepsilon\twodec\msdec)} 
  + 2\alpha_1;
\end{equation} the first case is when $d(x,y)\ge R_1$ in which we
have:
\begin{equation}
  \label{eq:moto8}
  \begin{split}
    \left| G_2(x) - G_2(y)\right| &\le \left|\hat g_1(x) - \hat
      g_1(y)\right| + 2\linfnrm\hat g_2.\\
    &\le\left(\sqrt{3} +
  C^2\frac{\varepsilon_0}{\hat\varepsilon\onedec\msdec}+\alpha_1\right)d(x,y),
  \end{split}
\end{equation}
and the second case is when $d(x,y)\le R_1$ in which case we have:
\begin{equation}
  \label{eq:moto9}
  \begin{split}
    \left| G_2(x) - G_2(y)\right| &\le \left|\hat g_2(x) - \hat
      g_2(y)\right| +
    \left|\hat g_1(x) - \hat
      g_1(y)\right|\\
 &\le \left(\sqrt{3} +
  C^2\frac{\varepsilon_0}{\hat\varepsilon\twodec\msdec}+2\alpha_1\right)d(x,y).
  \end{split}
\end{equation}
We now show that $G_2$ is asymptotically flat on $\hat J_2\cup
(J_1\gdec\cap\hat J_1)$. In fact, $\hat g_1$ is asymptotically flat on
$\hat J_1$ and hence on $\hat J_2\cup
(J_1\gdec\cap\hat J_1)$, and $\hat g_2$ is asymptotically flat on
$\hat J_2$ and vanishes on $(J_1\gdec)_{\tau_1}$. We finally establish
analogues of~(\ref{eq:moto2}) and (\ref{eq:moto4}). Pick $x\in
C_a\onedec$; then:
\begin{equation}
  \label{eq:moto10}
  \begin{split}
  \linfballnrm{\tang x, r_a\onedec. G_2 - h}. &\le \linfballnrm{\tang x, r_a. \hat
    g_1- h}. + \frac{2\linfnrm \hat g_2.}{r_a\onedec}\\
&\le \frac{2\hat\varepsilon\szdec\onedec}{r_a\onedec} +
C(\varepsilon\szdec\onedec+\sigma) + \alpha_1.
\end{split}
\end{equation}
Pick $x\in 
C_a\twodec$;    \pcreatenrm{linfcstball}{\infty,B(x,r_a\twodec)}{\|}then:
\begin{equation}
  \label{eq:moto11}
  \begin{split}
    \linfballnrm{\tang x, r_a\twodec. G_2 - h}.&\le 
 \frac{\linfcstballnrm\hat
      g_1 - \hat g_1(x).}{r_a\twodec}
    + \linfballnrm{\tang x, r_a\twodec. \hat
    g_2- h}.\\
&\le\alpha_1 + C(\varepsilon\szdec\twodec+\sigma) + 2\frac{\hat\varepsilon\szdec\twodec}{r_a\twodec}.
  \end{split}
\end{equation}
In connection with~(\ref{eq:moto11}) we observe that $\hat\varepsilon\szdec\twodec$ will be chosen later
to be insignificant next to $\min_{1\le a\le M_2}{r_a\twodec}$.
\makestep{4}{The construction of $G_{j+1}$, for $j\ge 2$}
\par We fix parameters $(\alpha_j,\varepsilon\jjdec\spdec)\in(0,1)^2$ 
and choose 
\begin{equation}
\label{eq:moto12}
R_j\le\alpha_j^2\min\{r_a\lldec: 1\le l\le j, 1\le a\le M_l\}
\end{equation}
such that $G_j$
is $\alpha_j$-Lipschitz on $\hat J_j$ below scale $R_j$. We then find
a compact set $J\jonehalfdec\subset\hat J_j\setminus J_j\gdec$ such that:
\begin{equation}
  \label{eq:moto13}
  \mu\left(
    (\hat J_j\setminus J_j\gdec) \setminus J\jonehalfdec
\right) \le \varepsilon\spdec\jjdec \mu(\hat J_j\setminus J_j\gdec),
\end{equation}
and find a $\tau_j>0$ such that $(J\jonehalfdec)_{\tau_j}\cap
(J_j\gdec)_{\tau_j}=\emptyset$.
\par We then construct $g_{j+1}$ and $\hat g_{j+1}$ as we did for
$g_2$ and $\hat g_2$ in \texttt{Step3}: one needs only to adjust the
indexes. We will refer to the variants of properties
\textbf{($g_2$:a)}--\textbf{($g_2$:c)} and \textbf{($\hat
  g_2$:a)}--\textbf{($\hat g_2$:c)} by \textbf{($g_{j+1}$:a)}--\textbf{($g_{j+1}$:c)} and \textbf{($\hat
  g_{j+1}$:a)}--\textbf{($\hat g_{j+1}$:c)}.
\par We then have to modify $\hat g_{j+1}$ so that it vanishes on
$(J_j\gdec)_{\tau_j}$ and stays the same on
$(J\jonehalfdec)_{\tau_j/2}$. This is accomplished by replacing $\hat
g_{j+1}$ with:
\begin{equation}
  \label{eq:moto14}
  \hat g_{j+1}\max\left(0, 1 - \frac{2}{\tau_j}\dist(\cdot, (J\jonehalfdec)_{\tau_j/2})\right);
\end{equation}
we also record the upper bound
\begin{equation}
  \label{eq:moto15}
  \glip\hat g_{j+1}. \le \sqrt{3} +
  C^2\frac{\varepsilon_0}{\hat\varepsilon\jonedec\msdec} 
  + 2\frac{\varepsilon\szdec\jonedec + \hat\varepsilon\szdec\jonedec}{\tau_j},
\end{equation}
and impose
the restriction $2(\varepsilon\szdec\jonedec +
\hat\varepsilon\szdec\jonedec)\le\alpha_j\tau_j$
to get the last term in~(\ref{eq:moto15}) to be $\le\alpha_j$.
\par We now let $G_{j+1} = G_j + \hat g_{j+1}$ and, akin to
(\ref{eq:moto7}), we obtain the upper bound:
\begin{equation}
  \label{eq:moto16}
  \glip G_{j+1}.\le \sqrt{3} + C^2\frac{\varepsilon_0}{\min_{l\le
      j+1}\hat\varepsilon\lldec\msdec} + 2\alpha_j,
\end{equation}
after imposing the restriction $2(\varepsilon\szdec\jonedec +
\hat\varepsilon\szdec\jonedec)\le\alpha_j R_j$.
\par Compared to \texttt{Step3}, the analogues of~(\ref{eq:moto10}),
(\ref{eq:moto11}) require some modifications because the errors
cumulate additively; keep also in mind that for us an empty sum like
$\sum_{j<1}\alpha_j$ defaults to $0$. For $x\in C_a\lldec$ where $l\le
j$ we obtain: \pcreatenrm{mylinfcstball}{\infty,B(x,r_a\lldec)}{\|}then:
\begin{equation}
  \label{eq:moto17}
  \begin{split}
    \linfballnrm{\tang x, r_a\lldec. G_{j+1} - h}.&\le 
 \frac{\mylinfcstballnrm G_{l-1} - G_{l-1}(x).}{r_a\lldec}
    + \linfballnrm{\tang x, r_a\lldec. \hat
    g_l- h}.\\
&+ \frac{\sum_{l < k\le j+1}\mylinfcstballnrm \hat g_k - \hat g_k(x).}{r_a\lldec}.
  \end{split}
\end{equation}
The first term in~(\ref{eq:moto17}) is bounded observing that
$G_{l-1}$ is $\alpha_{l-1}$ Lipschitz on $B(x, r_a\lldec)$ for $k<
l$. The second term is bounded using \textbf{($g_l$:c)} and
\textbf{($\hat g_l$:a)}. Finally the third term is bounded using
$\mylinfcstballnrm \hat g_k - \hat g_k(x).\le 2\|\hat g_k\|_\infty$
and minding that we imposed the restriction $2(\varepsilon\szdec\konedec +
\hat\varepsilon\szdec\konedec)\le\alpha_k R_k$. We thus get
from~(\ref{eq:moto17}):
\begin{equation}
  \label{eq:moto18}
  \begin{split}
    \linfballnrm{\tang x, r_a\lldec. G_{j+1} - h}.&\le \alpha_{l-1} +
    C(\varepsilon\szdec\lldec + \sigma) +
    2\frac{\hat\varepsilon\szdec\lldec}{r_a\lldec}
    +\sum_{l < k < j+1}\alpha_k\\
&=\sum_{\substack{l-1\le k < j+1\\ k\ne l}}\alpha_k +
C(\varepsilon\szdec\lldec + \sigma) +
2\frac{\hat\varepsilon\szdec\lldec}{r_a\lldec}.
\end{split}
\end{equation}
\par Pick $x\in C_a\jonedec$; then: \pcreatenrm{mmylinfcstball}{\infty,B(x,r_a\jonedec)}{\|}
\begin{equation}
  \label{eq:moto19}
  \begin{split}
    \linfballnrm{\tang x, r_a\jonedec. G_{j+1} - h}.&\le 
 \frac{\mmylinfcstballnrm G_{j} - G_{j}(x).}{r_a\jonedec}
    + \linfballnrm{\tang x, r_a\jonedec. \hat
    g_{j+1}- h}..
  \end{split}
\end{equation}
Using that $G_j$ is $\alpha_j$-Lipschitz on $B(x, r_a\jonedec)$ and
\textbf{($g_{j+1}$:c)}, \textbf{($\hat g_{j+1}$:a)} we finally get:
\begin{equation}
  \label{eq:moto20}
    \linfballnrm{\tang x, r_a\jonedec. G_{j+1} - h}.\le \alpha_j +
C(\varepsilon\szdec\jonedec + \sigma) +
2\frac{\hat\varepsilon\szdec\jonedec}{r_a\jonedec}.
\end{equation}

\makestep{5}{Choice of the parameters}
\par There are two kinds of parameters:
\begin{itemize}
\item Parameters that can be chosen arbitrarily small:
  $\varepsilon\szdec\lldec$, $\hat\varepsilon\szdec\lldec$,
  $\varepsilon\msdec\lldec$, $\varepsilon\spdec\lldec$ and $\alpha_l$.
\item Parameters that can't be chosen arbitrarily small:
  $\hat\varepsilon\msdec\lldec$ and $\sigma$. In particular, as
  $\hat\varepsilon\msdec\lldec\to0$ the Lipschitz constant of $G_l$ blows-up.
\end{itemize}
We thus choose:
\begin{equation}
  \label{eq:moto_21}
  \hat\varepsilon\msdec\lldec = 
  \begin{cases}
    \min\left( \frac{\varepsilon\msdec}{16}, \frac{1}{180C}
       \right) &\text{for $l\ge2$,}\\
    (\hat\varepsilon\msdec\twodec)^2&\text{for $l=1$.}
  \end{cases}
\end{equation}
We now estimate:
\begin{equation}
  \label{eq:moto_22}
  \mu(J_1\setminus J_1\gdec)\le\mu(J_1\setminus\hat J_1) + \mu(\hat
  J_1\setminus J_1\gdec)
  \le(\hat\varepsilon\msdec\onedec + 1 - C^{-1}\sigma^{N-N_0})\mu(J_1);
\end{equation}
thus, if we choose $\sigma = (1.5
C\hat\varepsilon\msdec\twodec)^{1/(N-N_0)}$ we get:
\begin{equation}
  \label{eq:moto_23}
  \mu(J_1\setminus J_1\gdec)\le\left( 1 - \frac{\hat\varepsilon\msdec\twodec}{2}
                               \right)\mu(J_1).
\end{equation}
For $l\ge 2$ the same argument for~(\ref{eq:moto_23}) yields:
\begin{equation}
  \label{eq:moto_24}
  \mu(J_l\setminus J_l\gdec)\le\left( 1 - \frac{\hat\varepsilon\msdec\twodec}{2}
                               \right)\mu(J_l);
\end{equation}
as $J_l\subset J_{l-1}\setminus J_{l-1}\gdec$ induction grants:
\begin{align}
  \label{eq:moto_25}
  \mu(J_l)&\le\left(1-\frac{\hat\varepsilon\msdec\twodec}{2}\right)^{l-1}\mu(K)\\
  \label{eq:moto_26}
  \mu(J_l\setminus J_l\gdec)&\le\left(1-\frac{\hat\varepsilon\msdec\twodec}{2}\right)^l\mu(K).
\end{align}
The construction in \texttt{Step 3} will be iterated finitely many
times and we just need an upper bound for the smallest number of
iterations which will give the desired approximation in measure (third
inequality in (a)). We get:
\begin{equation}
  \label{eq:moto_27}
  \begin{split}
    \mu\left(
         K\setminus\bigcup_{k=1}^l(\hat J_k\cap J_k\gdec)
       \right) &\le \mu(K\setminus J_1) +
       \sum_{k=1}^{l-1}\mu(J\konehalfdec\setminus
       J_{k+1})+\sum_{k=1}^l\mu(J_k\setminus\hat J_k)\\
       &+\sum_{k=1}^{l-1}\mu\left((\hat J_k\setminus
         J_k\gdec)\setminus J\konehalfdec\right) + \mu(J_l\setminus
       J_l\gdec)\\
       &\le\biggl(
         \varepsilon\msdec\onedec +
         \sum_{k=1}^{l-1}\varepsilon\msdec\konedec +
         \hat\varepsilon\onedec\sum_{k=1}^l\bigl(1-\frac{\hat\varepsilon\twodec}{2}\bigr)^{k-1}\\
         &+\sum_{k=1}^{l-1}\varepsilon\spdec\kkdec +
         \bigl(1-\frac{\hat\varepsilon\twodec}{2}\bigr)^{l}
         \biggr)\mu(K).
  \end{split}
\end{equation}
Recall that at each stage we can choose $\varepsilon\msdec\kkdec$ and
$\varepsilon\spdec\kkdec$ arbitrarily small and observe that:
\begin{equation}\label{eq:moto_28}
\hat\varepsilon\msdec\onedec\sum_{k=1}^l\bigl(1-\frac{\hat\varepsilon\twodec}{2}\bigr)^{k-1} \le2\frac{\hat\varepsilon\msdec\onedec}{\hat\varepsilon\msdec\twodec}=2\hat\varepsilon\msdec\twodec\le\varepsilon\msdec.
\end{equation}
Thus there is a universal constant $C_1$ such that if
$l:=\lceil|\log_{1-\hat\varepsilon\msdec\twodec/2}\varepsilon\msdec|\rceil$
one has:
\begin{equation}
  \label{eq:moto_29}
  \mu\left(
    K\setminus\bigcup_{k=1}^l(\hat J_k\setminus J_k\gdec)
     \right)\le C_1\varepsilon\msdec\mu(K).
\end{equation}
Thus we will let $J:=\bigcup_{k=1}^l(\hat J_k\setminus J_k\gdec)$. We
then have:
\begin{equation}
  \label{eq:moto_30}
  \|G_l\|_\infty\le 2\sum_{k=1}^l(\varepsilon\szdec\kkdec + \hat\varepsilon\szdec\kkdec)
\end{equation}
which can be made $\le C_1\varepsilon\szdec$ as the parameters
$\varepsilon\szdec\kkdec$, $\hat\varepsilon\szdec\kkdec$ can be chosen
arbitrarily small at each stage. Also the parameters $\alpha_k$ can be
chosen arbitrarily small; thus, as
$\hat\varepsilon\msdec\twodec\simeq\varepsilon\msdec^2$ we obtain a
universal constant $C_0$ such that:
\begin{equation}
  \label{eq:moto_31}
  \glip G_l.\le\left(\sqrt{3} + C_0\frac{\varepsilon_0}{\varepsilon\msdec^2}\right).
\end{equation}
We thus let $g:=G_l$ and~(\ref{eq:mina3}) now follows from
(\ref{eq:moto18}), (\ref{eq:moto20}) if we choose at each step the parameters
$\alpha_k$, $\varepsilon\kkdec\szdec$ and
$\hat\varepsilon\szdec\kkdec$ sufficiently small.
\end{proof}
%%%%%%%%%%%%%%%%%%%%%%%%%%%%%%%%%%%%%%%%%%%%%%%%%%%%%%%%%%%%%%%%%%%%%%%%%
\subsection{Proof of Theorem~\ref{main}(II)}
The proof will be achieved by an iteration of the Lemma \ref{lemmafinale}, below which is a simple consequence of Lemma \ref{lem:step-3}. More precisely, the function $g$ required in Definition \ref{def_presc} is obtained as a sum of functions $g_i$ given by Lemma \ref{lemmafinale}. One of the subtle points in the process is that in principle it could be $\glip g_i.\geq\sqrt 3$ for every $i$, hence the sum could fail to be Lipschitz in general. However, the fact that the $g_i$'s can be chosen asymptotically flat and with arbitrarily small norm, allows one to control the Lipschitz constant of the sum. 

%\begin{lem}\label{l:sumlip}
%Let $h_1:\Omega\subset\R^N\to \R$ be an $L$-Lipschitz function, which is asymptotically flat on a compact subset $K\subset\Omega$. Let $h_2:\Omega\to \R$ be an $L$-Lipschitz function. Then for every $\varepsilon>0$ there exists $\delta>0$ such that, if $\|h_2\|_{\infty}<\frac{\delta^2}{2}$, then $h_1+h_2$ is $(L+\varepsilon)$-Lipschitz on $K$.
%\end{lem}
%\begin{proof}
%Since $h_1$ is is asymptotically flat on $K$, then there exists $0<\delta\leq\varepsilon$ such that, whenever $x,y\in K$ are such that $d(x,y)<\delta$ one has 
%$$|h_1(x)-h_1(y)|\leq\varepsilon d(x,y).$$ 
%
%Consider now $x,y\in K$. If $d(x,y)<\delta$, it holds
%$$|(h_1+h_2)(x)-(h_1+h_2)(y)|\leq |h_1(x)-h_1(y)|+|h_2(x)-h_2(y)|\leq (\varepsilon+L)d(x,y).$$
%Otherwise, if $d(x,y)\geq\delta$,
%$$|(h_1+h_2)(x)-(h_1+h_2)(y)|\leq |h_1(x)-h_1(y)|+|h_2(x)-h_2(y)|$$
%$$\leq L d(x,y)+\delta^2\leq (L+\delta) d(x,y)\leq (L+\varepsilon) d(x,y).$$
%\end{proof}

\begin{lem}\label{lemmafinale}
  Let $f:\Omega\to {\rm{Lip}}(B,0)$ be a Borel map such that, for $\mu$-a.e. $x\in\Omega$, $f(x)\in C(\mu,x)$ and moreover the corresponding function $L$ in \eqref{e:admissible} vanishes. Let $K\subset\Omega$ be a compact set such that $\glip f(x).\leq 1$ for every $x\in K$ and let $\varepsilon>0$ be fixed. There are constants $C_0$ and $C_1$, depending only on $N$, a
  $(C_0+\sqrt 3)$-Lipschitz function 
  $g:\Omega\to\real$ and a compact $J\subset
  K$ such that:
  \begin{enumerate}[label={\normalfont(\alph*)}]
  \item $g$ is asymptotically flat on $J$,
    $\linfnrm g.\le \varepsilon$ and $\mu(K\setminus
    J)\le 3 C_1\varepsilon\mu(K)$.
  \item There are $0<r_1\le r_0\le\varepsilon$ and for every $x\in J$ there are $r_1\le r(x)\le r_0$ such that:
    \begin{equation}
      \label{eq:lemmafinale}
      \linfballnrm {\tang {x,r(x)}. g-f(x)}.\le 3C_1\varepsilon. 
    \end{equation}
  \item $g$ is supported on the tubular neighborhood of $K$ with radius $\varepsilon$.
  \end{enumerate}
\end{lem}

\begin{proof}

%We can assume without loss of generality that $\mu(\Omega)<\infty$. Indeed there are at most countably many open sets $U_\ell$ ($\ell=1,2,\ldots$) with 
%$$\inf_{k\neq \ell}\{dist(U_\ell,U_k)\}>0 \quad\text{for every } \ell$$ and such that $\mu(U_\ell)<\infty$ for every $\ell$ and $\mu(\Omega\setminus\bigcup_\ell U_\ell)<\frac{\varepsilon}{2}$. We will prove that for each $U_\ell$ there exists a function $g_\ell$ on $U_\ell$ as in Definition \ref{def_presc} (ii), with arbitrarily small norm and with $\glip g_\ell.$ bounded by a universal constant (depending on $\varepsilon$ and $f$). The fact that each $U_\ell$ has positive distance from the union of the other $U_k$'s, and the fact that the norm of $g_\ell$ can be chosen smaller than a multiple of such distance, allow to extend the function defined on $\bigcup_\ell g_\ell$ to the set $\Omega$ with a controlled Lipschitz constant. 

Let $C_0$ and $C_1$ be the constants in Lemma \ref{lem:step-3}. 
Since $f(x)\in C(\mu,x)$ for $\mu$-a.e. $x$, by the Lusin's theorem we can find at most $N+1$ disjoint compact sets $K_j\subset K$ $(j=0,\ldots, N)$ of positive measure, such that 
\begin{equation}\label{e:notte1}
\mu(K\setminus\bigcup_{j=0}^N K_j)<C_1\varepsilon\mu(K)
\end{equation}
 and $f(x)\in C(\mu,x)$ for every $x\in K_j$, for every $j$. Moreover, on each $K_j$ both $V(\mu,x)$ and $f(x)$ vary continuously in $x$ and $V(\mu,\cdot)$ has constant dimension $j$.\\

%For every $j$, consider the function $f^1_j:= K_j\to \rm{Lip}$ given by
%$$(f^1_j(x))(y):=(f(x))(y)-(\tilde L(x))(y).$$
%Fix $\varepsilon_1>0$. 

Since the Grassmannian of $j$-planes in $\R^N$ is totally bounded, and since $V(\mu,x)$ and $f(x)$ vary continuously in $x$ on each $K_j$, we can find finitely many disjoint non-empty compact subsets $K_j^\ell\subset K_j$ ($\ell=1,\ldots,k_j$) such that 

\begin{equation}\label{e:notte2}
\|V(\mu,x)-V(\mu,y)\|_\infty\leq\varepsilon^{2N}
\end{equation}
and 
\begin{equation}\label{e:notte3}
\|f(x)-f(y)\|_\infty\leq C_1\varepsilon
\end{equation}
for every pair $(x,y)$ of points in $K_j^\ell$ and moreover
\begin{equation}\label{e:notte4}
\mu(K_j\setminus \bigcup_\ell K_j^\ell)\leq\frac{C_1\varepsilon}{(N+1)}\mu(K).
\end{equation}

Choose for each $j,l$ a point $x_j^\ell$ and let $f_j^\ell:=f(x_j^\ell)\in \rm{Lip}$. Notice that, by assumption, $\glip f_j^\ell.\leq 1$. Let $$d:=\min\{1,\dist(K,(\R^N\setminus\Omega)),\min_{j,k}\{\dist(K_j^l,K_k^m):l\neq m\}\}.$$
For every $(j,\ell)$, apply Lemma \ref{lem:step-3} with $K:=K_j^\ell$, $\pi:=V(\mu,x_j^\ell)$, $h:=f_j^\ell\trace \pi^\perp$ and $\varepsilon_0:=\varepsilon^{2N},r_0:=\frac{d}{4}\varepsilon,\varepsilon\msdec:=\varepsilon^N,\varepsilon\szdec:=C_1^{-1}\varepsilon\min\{1,\frac{d}{4}\} $. 
By the lemma, for every $(j,\ell)$ there exist
a ($\sqrt{3}+C_0$)-Lipschitz function $g_j^\ell:\R^N\to\R$ and a compact set $J_j^\ell\subset K_j^\ell$ such that:
  \begin{enumerate}[label={\normalfont(\alph*)}]
  \item $g_j^\ell$ is asymptotically flat on $J^\ell_j$,
    $\linfnrm g^\ell_j.\le \varepsilon\min\{1,\frac{d}{4}\}$ and $\mu(K_j^a\setminus
    J_j^\ell)\le C_1\varepsilon^N(K_j^\ell)$.
  \item One can write every $J_j^\ell$ as a finite disjoint union
    $J_j^\ell = \bigcup_{a=1}^{M}J_{j,a}^\ell$ and
    for each $a\in{1,\cdots, M}$ ($M$ may depend on $j$ and $\ell$) there is an $0<r_a:=r_a(\ell,j)\le r_0$ such that if $x\in
    J_{j,a}^\ell$ one has:
    \begin{equation}
      \label{eq:apprbu}
      \linfballnrm {\tang.g^\ell_j-f}.\le 2C_1\varepsilon.
    \end{equation}
  \end{enumerate}
Via a simple cut-off, we can modify each $g_j^\ell$ to a ($\sqrt{3}+C_0$)-Lipschitz function $\bar g_j^\ell$ supported on $\Omega$ such that
$$\bar g_j^\ell=g_j^\ell \quad \text {on }\left\{x: \dist(x,K_j^\ell)<\frac{d}{4}\varepsilon\right\}$$
and
$$\bar g_j^\ell=0 \quad \text {on }\left\{x: \dist(x,K_j^\ell)>\frac{d}{2}\varepsilon\right\}$$

Finally we define the function $g:\Omega\to\R$ by
$$g:=\sum_{j,\ell}\bar g_j^\ell$$
and we observe that, by (a), $\|g\|_{\infty}\leq\varepsilon$.
Denoting $J:=\bigcup_{j,\ell}{J^\ell_j}$, by \eqref{e:notte1}, \eqref{e:notte4} and (a), we get 
$$\mu(K\setminus J)\leq 3C_1\varepsilon\mu(K).$$
Moreover, denoting $r_1:=\min_{j,\ell,a}\{r_a(\ell,j)\}$ and setting for every $x\in J_{j,a}^{\ell}$ $r(x):=r_a(\ell,j)$, we get by \eqref{e:notte3} and (b) that it holds $0<r_1\leq r(x)\leq r_0\leq\varepsilon$ and
    \begin{equation}
      \label{eq:apprbuBIS}
      \linfballnrm {\tang {x,r(x)}. g-f(x)}.\le 3C_1\varepsilon,
    \end{equation}
where we observe that by the choice of $r_0$, it holds $\tang {x,r(x)}. g=(g_j^\ell)_{x,r(x)}$, for every $x\in J_j^\ell$. 
\end{proof}

\begin{proof}[proof of Theorem \ref{main}(II)]

\makestep{1}{Prescribing the linear part}
Fix $\varepsilon>0$. It is not restrictive to assume that $\mu(\Omega)=1$. Let $K$ be a compact set such that $\mu(\Omega\setminus K)<\varepsilon/4$, $f(x)\in C(\mu,x)$, and $\glip f(x).\le D$, for every $x\in K$ and for some $D>0$. Without loss of generality we can assume that $D=1$. Firstly, for every $x\in K$ we extend the linear function $L$ which $f(x)$ defines on $V(\mu,x)$ (see \eqref{e:admissible}) to a linear function $\tilde L$ defined on $\R^N=(V(\mu,x),V(\mu,x)^{\perp})$ as 
$$\tilde L(x,y):=L(x).$$ 

%For every $j$ we denote by $\Omega_j\subset\Omega$ an open set containing $K_j$ such that 
%$$\mu(\Omega_j)\leq (1+\frac{\varepsilon}{2})\mu (K_j)$$
%and moreover
%$$d:=\min_{j\neq k}\{dist(\Omega_j,\Omega_k)\}>0.$$

Then we take any Borel measurable extension $\bar L$ of $\tilde L$ defined on the set $\Omega$ and preserving the bound $\glip \bar L.\leq 1$.
By Theorem \ref{mainI} applied to $f=\bar L$ and $\zeta=1$, we can find a compact set $K^0\subset K$ and a function $g_0\in C^1_c(\Omega)$ with $\|g_0\|_{\infty}\leq C$ such that 
$$\mu(\Omega\setminus K^0)<\varepsilon/2$$
$$Dg_0(x)=\bar L(x)=\tilde L, \quad \text{for every } x\in K^0$$
and $\glip g_0.\leq C$.
\\

\makestep{2}{Prescribing the non-linear part}
Consider the function $f_0:\Omega\to\Lip$ such that $f_0(x)\equiv 0$ for every $x\in\Omega\setminus K$ and
$f_0(x)=f(x)-\tilde L(x)$ for every $x\in K$.
We will apply Lemma \ref{lemmafinale} to a sequence of sets $(K_i)_{i\in\N}$ (with $K_0:=K$), the map $f_0$ and a sequence of parameters $(\varepsilon_i)_{i\in\N}$ with $\varepsilon_i\to 0$ and we will obtain respectively functions $g^i$, compact sets $J_i=:K_{i+1}$, and for every $x\in J_i$ a radius $r_1^i\leq r^i(x)\leq r^i_0\leq\varepsilon_i$.

Since we can choose the $\varepsilon_i$ inductively, we can assume that for every $i$ it holds
$$\varepsilon_i r_1^i\geq \sum_{j>i}\varepsilon_j $$
so that, for every $i$ it holds, for every $x\in J_i$
\begin{equation}
      \label{eq:non destroy1}
      \begin{split}
      &\linfballnrm {\tang x,r^i(x).(g^i+\sum_{j>i}g^j)-f_0(x)}.\\
      &\le\linfballnrm {\tang {x,r^i(x)} . g^i - f_0(x)}.\mskip 8mu+
      \linfballnrm {\tang x,r^i(x).\sum_{j>i}g^j}.\\
      &\le 3C_1\varepsilon^i + \frac{\sum_{j>i}\varepsilon_j
      }{r^i(x)} \leq (3C_1+1)\varepsilon_i. 
      \end{split}
    \end{equation}
Moreover, since the $g^j$'s are asymptotically flat on the $J_j$'s, we can add the further restriction on the inductive choice of the $\varepsilon_i$'s that, for every $j$ and for every $i>j$, $g_j$ is $(\varepsilon_j)^i$-Lipschitz on a tubular neighbourhood of $J_j$ of radius $\varepsilon_i$, below scale $\varepsilon_i$. Hence, since $r^i(x)\leq r_0^i\leq\varepsilon_i$, we have, for every $x\in J_i$
\begin{equation}
      \label{eq:non destroy2}
      \begin{split}
      &\linfballnrm {\tang x,r^i(x).(g^i+\sum_{j<i}g^j)-f_0(x)}.\\
      &\le\linfballnrm {\tang {x,r^i(x)}.g^i-f_0(x)}.\mskip 8mu+
      \linfballnrm {\tang x,r^i(x).\sum_{j<i}g^j}.\\
      &\le\linfballnrm {\tang {x,r^i(x)}.g^i-f_0(x)}.+
      \sum_{j<i}\glip g^j\trace B_{r^i(x)}(x).\\
      &\le 3C_1\varepsilon_i + \sum_{j<i}(\varepsilon_j)^i. 
      \end{split}
    \end{equation}

Denote $g_1:=\sum_{i\in\N} g^i$ and $J:=\cap_i K_i$. Combining \eqref{eq:non destroy1} and \eqref{eq:non destroy2} we have that, provided $(\varepsilon_i)_{i\in\N}$ respects the choices made above and provided $\sum_j(\varepsilon_j)^i\to 0$ as $i\to\infty$, and $\sum_j(\varepsilon_j)\leq\varepsilon/2$ it holds that $\mu(\Omega\setminus J)\leq \varepsilon/2$ and moreover, for every $x\in J$,
$$ \linfballnrm {\tang x,r^i(x).g_1-f_0(x)}.\to 0 \quad \text{ as } i\to\infty.$$

\makestep{3}{Lipschitz estimates} 
It remains to show that $g_1$ is Lipschitz. Let $x,y\in\Omega$ with $x\neq y$. Firstly observe that if $|x-y|\geq \frac{1}{2}\sum_i\varepsilon_i$, then the estimate is very simple, indeed
\begin{equation}\label{e:primastima} 
|g_1(y)-g_1(x)|\leq 2\sum_{i=1}^{\infty}\|g^i\|_\infty\leq 2\sum_i\varepsilon_i\leq 4|x-y|.
\end{equation}

Otherwise let us consider different cases. Here we make the following assumption on the sequence $(\varepsilon_j)_{j\in\N}$: for every $j$ it holds $\varepsilon_j\geq 2\varepsilon_k$ for every $k>j$, hence in particular $\varepsilon_j\geq \sum_{k>j}\varepsilon_k$.
\begin{itemize}
%\item If $x,y\in J$, let $i_0$ be the first index such that $\sum_{j=i_0}^{\infty}\varepsilon_j<|x-y|$. Since we can require that for every $j$ it holds $\varepsilon_j\geq\sum_{k>j}\varepsilon_k$, then necessarily $|x-y|<2\varepsilon_{i_0-1}\leq \varepsilon_{i_0-2}$. Since we know that for every $j$ and for every $i>j$, $g_j$ is $(\varepsilon_j)^i$-Lipschitz on the tubular neighbourhood of $B_{\varepsilon_i}(J_j)$, below scale $\varepsilon_i$, this implies that $\sum_{j=1}^{i_0-2}g^j$ is $(\sum_{j=1}^{i_0-2}(\varepsilon_j)^{i_0-2})$-Lipschitz on $J$ below scale $\varepsilon_{i_0-2}$ and in particular $|(\sum_{j=1}^{i_0-2}g^j)(y)-(\sum_{j=1}^{i_0-2}g^j)(x)|\leq |y-x|$. Hence 
%
%$$|g_1(y)-g_1(x)|\leq|(\sum_{j=1}^{i_0-2}g^j)(y)-(\sum_{j=1}^{i_0-2}g^j)(x)|+|g^{i_0-1}(y)-g^{i_0-1}(x)|$$
%$$+|(\sum_{j=i_0}^{\infty}g^j)(y)-(\sum_{j=i_0}^{\infty}g^j)(x)|$$
%$$\leq (1+\sqrt 3+C_0)|y-x|+2\sum_{j=i_0}^{\infty}\varepsilon_j\leq (3+\sqrt 3+C_0)|y-x|$$

\item if either $x\in J$ or $y\in J$. Let $j_0$ be the first index $j$ such that $|y-x|\geq\varepsilon_j$. In particular $x$ and $y$ are in $B_{\varepsilon_{j_0-1}}(J^{j_0-1})$. Since we know that for every $j<j_0-1$, the function $g^j$ is $(\varepsilon_j)^{j_0-1}$-Lipschitz on the tubular neighbourhood of $B_{\varepsilon_{j_0-1}}(J^j)$, below scale $\varepsilon_{j_0-1}$, this implies that $\sum_{j=1}^{j_0-2}g^j$ is $(\sum_{j=1}^{j_0-2}(\varepsilon_j)^{j_0-1})$-Lipschitz on $B_{\varepsilon_{j_0-1}}(J^{j_0-1})$ below scale $\varepsilon_{j_0-1}$ and in particular $|(\sum_{j=1}^{j_0-2}g^j)(y)-(\sum_{j=1}^{j_0-2}g^j)(x)|\leq |y-x|$. Moreover, since $\varepsilon_{j_0}\geq\sum_{k>j_0}\varepsilon_k$, it holds $|y-x|\geq\frac{1}{2}\sum_{j=j_0}^{\infty}\varepsilon_j$. Hence we can write
\begin{equation}\label{missing1}
\begin{split}
|g_1(y)-g_1(x)|&\leq|(\sum_{j=1}^{j_0-2}g^j)(y)-(\sum_{j=1}^{j_0-2}g^j)(x)|\\
&\mskip 8mu+|g^{j_0-1}(y)-g^{j_0-1}(x)|+|(\sum_{j=j_0}^{\infty}g^j)(y)-(\sum_{j=j_0}^{\infty}g^j)(x)|\\
&\leq (1+\sqrt 3+C_0)|y-x|+2\sum_{j=j_0}^{\infty}\varepsilon_j\leq (5+\sqrt 3+C_0)|y-x|.
\end{split}
\end{equation}
\item If $x\not\in J$ and $y\not\in J$, let $i_0$ (respectively $j_0$) be the first index $i$ such that $x\not\in B_{\varepsilon_i}(J^i)$
(respectively $y\not\in B_{\varepsilon_i}(J^i)$). We can assume, without loss of generality, that $i_0\leq j_0$. 

If $|x-y|<\varepsilon_{i_0+1}$, then necessarily $j_0-i_0\leq 1$. In
this case (recalling (c) in Lemma \ref{lemmafinale}) $g^j(x)=g^j(y)=0$
for every $j\geq i_0+1$. Moreover for every $i<i_0-1$, the function
$g^i$ is $(\varepsilon_i)^{i_0-1}$-Lipschitz on the tubular
neighbourhood of $B_{\varepsilon_{i_0-1}}(J^i)$, below scale
$\varepsilon_{i_0-1}$, hence, as in the previous case, we can estimate 
\begin{equation}\label{missing2}\begin{split}
|g_1(y)-g_1(x)|&\leq|(\sum_{i=1}^{i_0-2}g^i)(y)-(\sum_{i=1}^{i_0-2}g^i)(x)|+|g^{i_0-1}(y)-g^{i_0-1}(x)|
\\ &\mskip 8mu+|g^{i_0}(y)-g^{i_0}(x)|\\
&\leq \sum_{i=1}^{i_0-1}(\varepsilon_i)^{i_0}|y-x|+2(\sqrt 3+C_0)|y-x|\leq (1+2\sqrt 3+2C_0)|y-x|.
\end{split}
\end{equation}
The last case to analyse is when $|x-y|\geq\varepsilon_{i_0+1}$. In this case let $k_0\leq i_0+1$ be the first index $k$ such that $|x-y|\geq \varepsilon_k$. Note that if $k_0=1$, then we fall in the first case considered, because we have required in particular that $\varepsilon_1\geq\sum_{j>1}\varepsilon_j$, and hence \eqref{e:primastima} provides the Lipschitz estimate. Therefore we can consider only the case $k_0\geq 2$. 

Since $|x-y|\leq \varepsilon_{k_0-1}$, then for every $i<k_0-1$, the
function $g^i$ is $(\varepsilon_i)^{k_0-1}$-Lipschitz on the tubular
neighbourhood of $B_{\varepsilon_{k_0-1}}(J^i)$, below scale
$\varepsilon_{k_0-1}$, hence
\begin{equation}\label{missing3}\begin{split}
|g_1(y)-g_1(x)|&\leq|(\sum_{j=1}^{k_0-2}g^j)(y)-(\sum_{j=1}^{k_0-2}g^j)(x)|+|g^{k_0-1}(y)-g^{k_0-1}(x)|
\\ &\mskip 8mu+|g^{k_0}(y)-g^{k_0}(x)|+2\sum_{j=k_0+1}^{\infty}\varepsilon_j \\
&\leq(\sum_{j=1}^{k_0-2}(\varepsilon_j)^{k_0-1}+2\sqrt 3
+2C_0)|y-x|+2\varepsilon_{k_0}\\
&\leq(4+2\sqrt 3 +2C_0)|y-x|.
\end{split}
\end{equation}

%
%
%We can immediately find a simple Lipschitz estimate if $i_0<3$. Indeed in this case (recalling (c) in Lemma \ref{lemmafinale}) $g^j(x)=0$ for every $j>1$. Hence if $j_0<4$, in which case $g^j(y)=0$ for every $j>2$ and hence $g_1$ is both at $x$ and at $y$ the sum of at most 2 non null Lipschitz functions. Otherwise If $j_0\geq 4$, necessarily $|x-y|\geq\frac{\varepsilon_1}{2}\geq\frac{1}{4}\sum_i\varepsilon_i$ and in this case the Lipschitz esimate is provided by \eqref{e:primastima}.
%
%We can therefore consider only the case when $i_0\geq 3$. We firstly observe that, by (c) in Lemma \ref{lemmafinale} it holds $g^j(x)\equiv 0$ for $j\geq i_0$. Moreover, if $|x-y|<\varepsilon_{i_0-1}$, for every $i<i_0-1$, the function $g^i$ is $(\varepsilon_i)^{i_0-1}$-Lipschitz on the tubular neighbourhood of $B_{\varepsilon_{i_0-1}}(J^i)$, below scale $\varepsilon_{i_0-1}$, hence, as in the previous case, we can estimate
%%If $j_0-i_0\geq 2$, then $|x-y|>\varepsilon_{i_0+1}$,  
%$$|g_1(y)-g_1(x)|\leq|(\sum_{i=1}^{i_0-2}g^j)(y)-(\sum_{i=1}^{i_0-2}g^j)(x)|+|g^{i_0-1}(y)-g^{i_0-1}(x)|+\sum_{i\geq i_0}|g^{i}(y)|$$ 
%$$\leq \sum_{i=1}^{i_0-2}(\varepsilon_i)^{i_0-1}|y-x|+(\sqrt 3+C_0)|y-x|+2\varepsilon_{i_0}\leq (1+\sqrt 3+C_0)|y-x|+2\varepsilon_{i_0+1}$$ 
%
\end{itemize}

\makestep{4}{Conclusion of the proof}
Consider the Lipschitz function $g:=g_0+g_1$. It is easy to see that it holds
$$\mu(\{x\in\Omega:f(x)\not\subset \Tan(g,x)\})\leq \mu(\Omega\setminus(K^0\cap J))<\varepsilon.$$
Indeed $d_{V(\mu,x)}g_1=0$ on $J$ and $d_{V^{\perp}(\mu,x)}g_0=0$ on $K^0$.
Hence, since $\varepsilon$ can be chosen arbitrarily small, $f$ prescribes the blowups of a Lipschitz function weakly in the Lusin sense.

\end{proof}
\begin{remark}\rm{
As we will show in the proof of Theorem \ref{main1} (II), once it is possible to prescribe weakly (in the Lusin sense) a blowup in a closed class of admissible functions, it is also possible to prescribe all the admissible blowups ``simultaneously''. By this we mean that it is possible to find a Lipschitz function attaining all admissible blowups on an arbitrarily large set of points. In the previous proof it is sufficient to select (in a measurable way) a countable dense set of admissible blowups $\{g_i(x)\}$ for every point $x$, and, selecting in a suitable way different blowups at different scales, one can build a function $f$ attaining at many points all the $g_i$'s as blowups. To conclude, it is sufficient to observe that the set of all blowups at one point is closed.}
\end{remark}
\section{Optmality of the class $C(\mu,\cdot)$}\label{s6}
We now give an example of a measure for which one cannot prescribe
more blowups than those contained in $C(\mu,\cdot)$. In general it
seems a hard problem to characterize the largest set of blowups one can prescribe in terms of structural properties of the Radon measure.\\
\def\tang#1.{\setbox1=\hbox{$#1$\unskip}
{\normalfont\text{T}}_{\ifdim\wd1>0pt #1\else x,r\fi}}

Given a Radon measure $\mu$ on $\R^N$, $r>0$ and a point $x$ we define
the measure $\tang .\mu$ on $B$ by
$$\tang.\mu(A):=\mu(x+rA), \;{\rm{for\,every\;Borel\;set\;}}A\subset B.$$ 
We denote by $\Tan(\mu,x)$ the set of the \emph{blowups} of $\mu$ at $x$, i.e. all the possible limits of the form
$$\lim_{r_i\searrow 0}\kappa_i$$
where
\begin{equation}\label{eqbu}
\kappa_i:=\frac{\tang x,r_i.\mu}{\mu(B(x,r_i))}.
\end{equation}

Fix $k\in\{1,\dots,N-1\}$ and let $\nu_\R$ be a (doubling) Radon measure
on $\R$ such that $\nu_\R$ is singular with respect to the Lebesgue
measure, its support is $\R$ and for $\nu_\R$-a.e $x\in\R$ the set $\Tan(\nu_\R,x)$ contains only the Lebesgue measure. Examples of such measures are discussed in \cite{preiss-meageo} or can be obtained modifying the example of \cite{schul-doubling-conc}. 

Let $\mu$ be the product measure 
$$\mu:=\Leb^k\otimes \underbrace{\nu_\R\otimes\ldots\otimes\nu_\R}_{(N-k)-times}$$
and consider a Lipschitz function $f$ on
$\R^N=\R^k\times\R^{N-k}$. The decomposability bundle $V(\mu,\cdot)$
coincides with $\R^k$ as $\mu$, by the properties of $\nu_\R$, is
concentrated on a set which intersects each $C^1$-curve $\gamma$ whose tangent vector does not
lie in $\R^k$ in a set of zero $1$-dimensional Hausdorff measure, and
thus we have a well defined derivative $d_{\R^k}f$ in the direction of $\R^k$. Observe that $\mu$ is also doubling and fix a point $P$ of approximate continuity for $d_{\R^k}f$. We can assume that at $P$ all blowups of $\mu$ are positive multiples of $\Leb^N$.

Let $g$ be a blowup of $f$ at $P$ and let $(r_i)_{i\in\N}$ be a
sequence of radii for which $g=\lim_{i\to\infty} \tang P,r_i. f$. Let $\tilde\mu$ be the limit of any converging subsequence of $\kappa_i$, defined in \eqref{eqbu} with $x=P$. Since the support of $\tilde \mu$ is the whole ball $B$ and since $d_{\R^k}f$ is approximately continuous at $P$, for every $q\in\R^N$ and $v\in\R^k$ such that $q+v\in B$ we get
$$g(q+v)-g(q)=\langle d_{\R^k}f(P),v\rangle.$$
Let $m$ be the (Lipschitz) restriction of $g$ to $\R^{N-k}\cap B$. Then for every $x\in\R^k, y\in\R^{N-k}$ such that $x+y\in B$ it holds
$$g(x,y)=m(y)+\langle d_{R^k}f(P),x\rangle,$$
hence $g\in C(\mu,P)$.

\section{Proof of theorem \ref{main1}}\label{s4}
As we already observed in point (iii) of Remark \ref{rmk1}, statement (I) of Theorem \ref{main1} is contained in Proposition 4.2 of \cite{Ma}. Regarding statement (II), we will prove a stronger (perhaps surprising) statement: namely we will prove that if $\mu$ is singular, then the generic $1$- Lipschitz function (in the sense of Baire categories) attains every 1-Lipschitz function as blowup at $\mu$-almost every point.

In this section we denote by $X$ the complete metric space of 1-Lipschitz functions on $\R$ endowed with the supremum norm. By $\mu$ we denote a singular probability measure on $\R$. We begin with the following lemma. 

\begin{lem}[Covering by intervals with non-negligible centres]
\label{cor:large_centers}
Let $U\subset\R$ be an open set. For every $r_0>0$ and $n\in\N$ there is a sequence of closed intervals $\{[x_j-r_j, x_j+r_j] \}_j$ contained in $U$ with disjoint interiors such that:
\begin{align}
  \label{eq:large_centers_cs1}
  r_j&\le r_0\\  \label{eq:large_centers_cs2}
 \mu(\bigcup_j[x_j-(8n)^{-1}r_j, x_j+(8n)^{-1}r_j])&\geq \frac{1}{16}n^{-1}\mu(\bigcup_j[x_j-r_j, x_j+r_j])\\  \label{eq:large_centers_cs3}
 \mu(U\setminus\bigcup_j[x_j-r_j, x_j+r_j])&=0.
\end{align}
\end{lem}
%\begin{proof}
%\end{proof}
\begin{proof}
Given any measure $\mu$ on $\R$ and $0<\lambda<1$, it is known that for $\mu$-a.e. $x$ there holds
$$\limsup_{r\to 0}\frac{\mu(\overline{B(x,\lambda r)})}{\mu(\overline{B(x,r)})}\geq \lambda.$$
The lemma then follows from Besicovitch Covering Theorem. We thank the referee for suggesting this short proof. We also show a constructive proof, in order to help the reader to understand the ideas behind Corollary \ref{cor:good_rect}.

Through the proof the closed interval $[x-r,x+r]$ will be denoted by $I(x,r)$. Firstly we apply Corollary \ref{cor:neg_frames} with $\varepsilon=2^{-6}$ obtaining a sequence of disjoint intervals $\{I(z_\lambda,r_\lambda)\}_\lambda$. 
By the choice of $\varepsilon$, for every $\lambda$ it holds that 
\begin{equation}\label{e:fava1}
\mu (I(z_\lambda,(1-2^{-6})r_\lambda)\geq\frac{1}{2}\mu(I(z_\lambda, r_\lambda)).
\end{equation}
Now we ``split'' each interval $I(z_\lambda, (1-2^{-6})r_\lambda), $ into $2^7-2$ sub-intervals
$$\{I_\lambda^i:=I(z_\lambda^i, 2^{-7}r_\lambda)\}_{i=1}^{2^7-2}.$$ 
with disjoint interiors and length $2^{-6}r_\lambda$. Denote by $\bar I_\lambda^i$ the ``central part'' of $I_\lambda^i$, i.e.
$$\bar I_\lambda^i:=I(z_\lambda^i, 2^{-10}n^{-1}r_\lambda).$$
Observe that, for every $\lambda$, the family
$$\left\{\bigcup_{i=1}^{2^7-2}\bar I_\lambda^i+j2^{-9}n^{-1}r_\lambda\right\}_{j=0,\ldots,8n-1}$$
covers the set $I(z_\lambda, (1-2^{-6})r_\lambda)$.
Hence for at least one index $j_0$, the set 
\begin{equation}\label{e:fava2}
\bigcup_{i=1}^{2^7-2}I_\lambda^i+j_02^{-9}n^{-1}r_\lambda
\end{equation}
satisfies
$$\mu(\bigcup_{i=1}^{2^7-2}\bar I_\lambda^i+j_02^{-9}n^{-1}r_\lambda)\geq \frac{1}{8}n^{-1}\mu(I(z_\lambda, (1-2^{-6})r_\lambda))\stackrel{\eqref{e:fava1}}{\geq} \frac{1}{16}n^{-1}\mu(I(z_\lambda, r_\lambda)).$$
Moreover, for every $j=0,\ldots,8n-1$ and every $i=1,\ldots,2^7-2$ the interval $I_\lambda^i+j2^{-9}n^{-1}r_\lambda$ is contained in the interior of $I(z_\lambda, r_\lambda)$
The result follows by adding to these intervals the two intervals 
$$[z_\lambda-r_\lambda, a] \quad\text{and}\quad [b, z_\lambda+r_\lambda],$$
where $a$ and $b$ are respectively the minimum and the maximum of the set in \eqref{e:fava2}.
Of course the procedure above should be also repeated for every $\lambda$.

\end{proof}

\begin{prop}\label{p:baire}
Let $\mu$ be a singular probability measure on $\R$. Let $f:[-1,1]\to\R$ be a 1-Lipschitz function with $f(0)=0$. Then the set $$X_f:=\{g\in X: f\in \Tan(g,x)\; {\rm{for\;\mu-almost\;every\;}}x\}$$
is residual in $X$ (i.e. it contains the intersection of countably many open and dense sets).
\end{prop}
\begin{proof}
For $n\in\N$ and $g\in X$, consider the set
$$E^n_g:=\{x\in\R:\exists\; \rho< n^{-1}\; {\rm{s.t.}}\;|f-\tang {x,\rho}. g|<n^{-1}\}.$$
First of all we notice that $E^n_g$ is open. Indeed if $x\in E^n_g$, $\rho\leq n^{-1}$ satisfies $|f-\tang {x,\rho}. g|<n^{-1}$ and $y\in\R$ is so that 
$$2|y-x|< \rho(n^{-1}-|f-\tang {x,\rho}. g|),$$ then, using that $g$ is 1-Lipschitz, we deduce 
$$|f-\tang {y,\rho}. g|\leq |f-\tang {x,\rho}. g|+|\tang {x,\rho}. g-\tang {y,\rho}. g|\leq |f-\tang {x,\rho}. g|+2\rho^{-1}|y-x|<n^{-1},$$
hence $y\in E^n_g$. Now we define
$$A_n:=\{g\in X: \mu(E^n_g)> 1-n^{-1}\}.$$\\

\makestep{1}{$A_n$ is open}
Fix $g\in A_n$ and consider the multifunction $\varrho:E^n_g\to 2^{(0,n^{-1})}$ defined by 
$$x\mapsto\{\rho\in(0,n^{-1})\;{\rm{s.t.}}\;|f-\tang {x,\rho}. g|<n^{-1}\}.$$ Notice that the values of $\varrho$ are non-empty open sets because the function $(x,\rho)\mapsto |f-\tang {x,\rho}. g|$ is continuous in the variable $\rho$. Moreover, since the function is continuous also in the variable $x$, for $\delta>0$ the sets 
\begin{multline*}
U_\delta:=\{x\in E_g^n: \rho_0(x):=\sup\{\varrho(x)\}>\delta\text{ and there exists }\rho(x)\text{ s.t. }\\
\delta<\rho(x)\in\varrho(x) \text{ and } |f-\tang {x,\rho(x)}. g|<n^{-1}-\delta\}
\end{multline*}
are open and $\bigcup_{\delta>0}U_\delta=E_g^n$. Moreover $\mu(E_g^n)>1-n^{-1}$, since $g\in A_n$. Then there exists $\delta>0$, with $\mu(U_\delta)>1-n^{-1}$. 
%Since $E^n_f$ is open, then there exists $\varepsilon\leq \rho_0$ such that $\mu(E_f^n\setminus E_{\varepsilon})<\frac{\delta}{2}$, where
%$$E_{\varepsilon}:=\{x\in E^n_f: d(x,\R\setminus E^n_f)\geq\varepsilon\}.$$

If we consider now $h\in X$ such that $2|g-h|<\delta^2$, we deduce that for every $x\in U_\delta$ it holds
\begin{multline*}
|f-\tang x,\rho(x).h|\leq |f-\tang {x,\rho(x)}. g|+|\tang {x,\rho(x)}. g-\tang {x,\rho(x)}. h|\\
\leq |f-\tang {x,\rho(x)}. g|+2|g-h|\delta^{-1}< (n^{-1}-\delta)+\delta.
\end{multline*}
The last inequality guarantees that $A_n$ is open.\\

\makestep{2}{$A_n$ is dense} Let $g\in X$ and fix $\varepsilon>0$. We want to show that there exists $h\in A_n$ such that $|h-g|\leq\varepsilon$. 

Consider inductively a sequence of functions $h_i$ defined as follows. Let $h_0:=g$, $M_0:=\R$, $\alpha_0:=\varepsilon$ and for $i=1,2,\ldots$ let $U_i\subset M_{i-1}$ be an open set such that 
\begin{equation}\label{e:bound1}
\Leb^1(U_i)\leq\frac{\alpha_{i-1}}{16n}
\end{equation} 
and 
\begin{equation}\label{e:bound2}
\mu(M_{i-1}\setminus U_i)<\frac{1}{n2^{i+2}}.
\end{equation} 
Moreover by Corollary \ref{cor:large_centers} we can select $I^i_1,\ldots, I^i_{m(i)}\subset U_i$ closed intervals with center $x^i_j$, length $4\ell^i_j$ and disjoint interiors such that firstly
\begin{equation}\label{e:bound3}
\mu(U_i\setminus\bigcup_{j=1}^{m(i)}I^i_j)<\frac{1}{n2^{i+2}}
\end{equation}
and secondly, denoting $\bar{I}^i_j$ the closed interval with center $x^i_j$ and length $(2n)^{-1}\ell^i_j$,
\begin{equation}\label{e:bound4}
\mu(\bigcup_{j=1}^{m(i)}\bar{I}^i_j)\geq (16n)^{-1}\mu(\bigcup_{j=1}^{m(i)}I^i_j),\quad {\rm{for\; every}}\; i.
\end{equation}

Or first aim is to perturb $h_{i-1}$ obtaining a new function $h_i$ such that all the points in $\bigcup_{j=1}^{m(i)}\bar I^i_j$ belong to $E^n_{h_i}$. Let $f_i$ be the 1-Lipschitz function
$$f_i(x):=h_{i-1}(x-|(-\infty,x)\cap\bigcup_{j=1}^{m(i)}I^i_j|).$$
Observe that $f_i$ is differentiable with $f_i'\equiv 0$ on the set $\bigcup_{j=1}^{m(i)}I^i_j$.
Denote by $k_i$ the 1-Lipschitz function
$$k_i(x):=f_i(x)+\sum_{j=1}^{m(i)} f_i^j(x),$$
where $f_i^j:\R\to\R$ is any 1-Lipschitz function such that $f_i^j\equiv 0$ on $\R\setminus I^i_j$ and
\begin{equation}\label{e_fix_bu}
\tang {x^i_j,\ell^i_j}.f_i^j=f
\end{equation}
(observe that such a function exists because $f$ is 1-Lipschitz, $g(0)=0$ and the length of the interval $I^i_j$ is $4\ell^i_j$). Eventually we define the 1-Lipschitz function $h_i:=f_i+k_i$.

Denote, for every $i$ 
$$\alpha_i:=\min_{j=1,\ldots,m(i)}\{\ell^i_j\};\quad M_i:=(\bigcup_{j=1}^{m(i)}I^i_j)\setminus(\bigcup_{j=1}^{m(i)}\bar{I}^i_j).$$

Note that the following properties hold, for every $i,j$
\begin{itemize}
\item[(i)] $|h_i-h_{i-1}|\leq 2\Leb^1(U^i)\stackrel{\eqref{e:bound1}}{\leq}\frac{\alpha_{i-1}}{8n}$,
\item[(ii)] $|\tang {x,\ell^i_j}. h_i-f|<(2n)^{-1}$, for every $x\in\bar{I}^i_j$.
\end{itemize}

Comparing (ii) with the definition of $E_g^n$, it is evident not only
that every $x\in\bar{I}^i_j$ belongs to $E_{h_i}^n$, but also that
there is still ``room'' for some additional perturbation. Namely for
every function $\tilde h$ with $|h_i-\tilde h|<(4n)^{-1}\alpha_i$ it
holds that every $x\in\bar{I}^i_j$ belongs to $E_{\tilde h}^n$, indeed
\begin{equation}\label{andreas-crap1}\begin{split}
|\tang {x,\ell^i_j}. \tilde h-f|&\stackrel{(i)}\leq|\tang
{x,\ell^i_j}. \tilde h- \tang {x,\ell^i_j}. h_i|+|\tang {x,\ell^i_j}. h_i-f|\\
 &\stackrel{(ii)}\leq 2(\ell^i_j)^{-1}|\tilde h-h_i|+
 (2n)^{-1}<(2n)^{-1}(\ell^i_j)^{-1}\alpha_i+(2n)^{-1}\\
 &\leq n^{-1}.
 \end{split}
 \end{equation}
In particular, for every $i,j$, for every $x\in\bar{I}^i_j$ and for every $m>i$ it holds $x\in E^n_{h_m}$, since 
$$|h_i-h_m|\leq \sum_{j=i}^{m-1}\frac{\alpha_j}{8n}\stackrel{\eqref{e:bound1}}\leq\frac{1}{8n}\alpha_i\sum_{j=0}^{m-j-1}(16)^{-i}<(4n)^{-1}\alpha_i.$$

Moreover, by (i) and the choice of $\alpha_0$ it follows that $|g-h_m|<\varepsilon$, for every $m$.
Combining \eqref{e:bound2}, \eqref{e:bound3} and \eqref{e:bound4} we deduce that for $i_0$ large enough we have $$\mu(\bigcup_{i\leq i_0}(\bigcup_{j=1}^{m(i)}\bar{I}^i_j))>1-n^{-1},$$ hence denoting $h:=h_{i_0}$, we have that $h\in A_n$.\\

\makestep{3}{Conclusion of the proof}
Clearly every function which belongs to the intersection of the $A_n$'s is also in $X_f$, hence $X_f$ is a residual set and in particular, by the Baire theorem, it is dense in $X$.
\end{proof}

\begin{proof}[Proof of theorem \ref{main1}(II)]
Without loss of generality we can assume that the Lipschitz constant of $f(x)$ is bounded by 1 for $\mu$-a.e. $x$ and that $\Omega=\R$, because Proposition \ref{p:baire} also holds when $\R$ is replaced by an open subset $\Omega$ (clearly in this case the space $X$ will be replaced by the space $X^\Omega$ of 1-Lipschitz functions on $\Omega$). 

{\bf{First case.}} $\mu$ is a finite measure.
Consider the metric space $Z$ made by the 1-Lipschitz functions on $[-1,1]$ with value 0 at the origin, endowed with the supremum distance. Let $(f_i)_{i\in\N}$ be dense in $Z$. Up to rescaling, we may assume that $\mu$ is a probability measure. By Proposition \ref{p:baire} each set $X_{f_i}$ is residual in $X$, and so it is $Y:=\bigcap_i{X_{f_i}}$. This means that for all $g\in Y$ and for $\mu$-a.e. $x$, every $f_i$ belongs to $\Tan(g,x)$. The theorem is then a consequence of the simple observation that $\Tan(g,x)$ is always a closed subset of $Z$.

{\bf{Second case.}} $\mu$ is any Radon measure. Write $\R$ as a countable union of sets $E_i$, $i=1,2,\ldots$ with finite measure. Consider for every $i$ the space $Y_i$ defined above relatively to the measure $\mu\trace E_i$. Since each $Y_i$ is residual in $X$, so it is the set $Y_\infty:=\cap_i Y_i$.
\end{proof}

%%%%%%%%%%%%%%%%%%%%%%%%%%%%%%%%%%%%%%%%%%%%%%%%%%%%%%%%%%%%%%%%%
%
%	BIBLIOGRAPHY
%
%%%%%%%%%%%%%%%%%%%%%%%%%%%%%%%%%%%%%%%%%%%%%%%%%%%%%%%%%%%%%%%%%
%%%%%%%%%%%%%%%%%%%%%%%%%%%%%%%%%%%%%%%%%%%%%%%%%%%%%%%%%%%%%%%%%%%%%%%%%%
\bibliographystyle{alpha}
\bibliography{prescribed_blowups_biblio}
%%%%%%%%%%%%%%%%%%%%%%%%%%%%%%%%%%%%%%%%%%%%%%%%%%%%%%%%%%%%%%%%%%%%%%%%%%

%%%%%%%%%%%%%%%%%%%%%%%%%%%%%%%%%%%%%%%%%%%%%%%%%%%%%%%%%%%%%%%%%
%
%	AFFILIATIONS
%
%%%%%%%%%%%%%%%%%%%%%%%%%%%%%%%%%%%%%%%%%%%%%%%%%%%%%%%%%%%%%%%%%

\vskip .5 cm

{\parindent = 0 pt\begin{footnotesize}

A.M.
\\
Institut f\"ur Mathematik,
Mathematisch-naturwissenschaftliche Fakult\"at,
Universit\"at Z\"urich\\
Winterthurerstrasse 190,
CH-8057 Z\"urich,
Switzerland
\\
e-mail: {\tt andrea.marchese@math.uzh.ch}

\end{footnotesize}
}
\vskip .5 cm
{\parindent = 0 pt\begin{footnotesize}

A.S.
\\
Department of Mathematics,
ETH Z\"urich\\
R\"amistrasse 101
CH-8092 Z\"urich,
Switzerland
\\
e-mail: {\tt andrea.schioppa@math.ethz.ch}

\end{footnotesize}
}

\end{document}